\documentclass[11pt]{amsart}
\usepackage{bbm}
\usepackage{xcolor}
\usepackage[normalem]{ulem}
\usepackage{yfonts}

\usepackage[bookmarks,colorlinks]{hyperref}
\usepackage{amssymb}
\theoremstyle{plain}
\newtheorem{thm}{Theorem}[section]

\newtheorem{lem}[thm]{Lemma}
\newtheorem{prop}[thm]{Proposition}
\newtheorem{cor}[thm]{Corollary}

\theoremstyle{definition}
\newtheorem{defn}{Definition}[section]

\theoremstyle{remark}
\newtheorem*{rem*}{Remark}
\newtheorem{rem}[thm]{Remark}

\newcommand{\R}{\mathbb{R}}
\newcommand{\Rd}{{\R^d}}
\newcommand{\Rn}{{\R^n}}
\newcommand{\Sn}{\mathbb{S}^{n-1}}

\newcommand{\Lp}{L^p(\Rd)}

\newcommand{\D}{\mathcal{D}}

\newcommand{\DpL}{\D_p(L)}

\newcommand{\Ex}{{\mathbb{E}_x}}
\newcommand{\E}{\mathcal{E}}
\newcommand{\DEp}{\D(\E_p)}
\newcommand{\intRd}{\int_\Rd}
\newcommand{\dx}{{\rm d}x}
\newcommand{\dy}{{\rm d}y}
\newcommand{\dz}{{\rm d}z}
\newcommand{\dt}{{\rm d}t}
\newcommand{\mP}{P}

\newcommand{\indyk}{\mathbf{1}}

\DeclareMathOperator*\sgn{sgn}
\DeclareMathOperator*{\esssup}{ess\,sup}

\makeatletter
\@addtoreset{equation}{section}
\makeatother

\date{\today}
\author[K.~Bogdan]{Krzysztof Bogdan}
\address{Faculty of Pure and Applied Mathematics, Wroc\l{}aw University of Science and Technology, Wyb. Wyspia\'nskiego 27, 50-370 Wroc\l{}aw, Poland.}
\email{krzysztof.bogdan@pwr.edu.pl}
\author[M.~Gutowski]{Micha\l\ Gutowski}
\address{Faculty of Pure and Applied Mathematics, Wroc\l{}aw University of Science and Technology, Wyb. Wyspia\'nskiego 27, 50-370 Wroc\l{}aw, Poland.}\email{michal.gutowski@pwr.edu.pl}
\author[K.~Pietruska-Pa\l uba]{Katarzyna Pietruska-Pa\l uba}
\address{Institute of Mathematics, University of Warsaw, ul. Banacha 2, 02-097 Warsaw, Poland.}
\email{kpp@mimuw.edu.pl}

\thanks{The research was supported by the  NCN grant 2018/31/B/ST1/03818.}

\subjclass[2010]{Primary 46E35; Secondary 31C05}

\keywords{Bregman co-divergence, Markovian semigroup, calculus in $L^p$}

\begin{document}
	\title[Polarized Hardy\nobreakdash--Stein identity]{Polarized Hardy\nobreakdash--Stein identity}

	\begin{abstract} We prove the Hardy\nobreakdash--Stein identity for  vector functions in $L^p(\Rd;\R^n)$ with $1<p<\infty$ and for the canonical paring of two real functions in $L^p(\Rd)$ with $2\le p<\infty$. To this end we propose a notion of Bregman co-divergence and study the corresponding integral forms.
	\end{abstract}\
	\maketitle
	
	\section{Introduction}\label{s.i}
	Hardy\nobreakdash--Stein identity is a disintegration of \textit{energy} 
	of a function
	in an elliptic or parabolic setting.
	The energy 
	is defined as an integral of $|f|^p$ and the
	disintegration is expressed 
	in terms of the corresponding \textit{Bregman divergence}.
	This is exemplified by the following result from Ba\~{n}uelos, Bogdan, and Luks \cite[Theorem~3.2]{MR3556449}:
	If $p>1$ and $f\in\Lp$ then
	\begin{align}
		\label{eq:HS-BBL}
		\intRd |f(x)|^p \,\dx = \int_0^\infty \int_\Rd \int_\Rd F_p(P_tf(x),P_tf(x+y))\, \nu(\dy)\dx\dt.
	\end{align}
	Here $F_p(a,b):=|b|^p-|a|^p-pa|a|^{p-2}(b-a)\ge 0$, $a,b\in \R$, is the Bregman divergence and $(P_t)_{t\geq 0}$ is the convolution semigroup on $\Rd$ generated by symmetric L\'evy measure $\nu$ satisfying the Hartman\nobreakdash--Wintner condition. We can rewrite
 \eqref{eq:HS-BBL} as 
	\begin{align}
		\label{eq:HS}
		\intRd |f(x)|^p \,\dx = p \int_0^\infty \E_p[P_t f] \,\dt,
	\end{align}
	where, for $u\!:\Rd\to \R$, the \textit{Sobolev-Bregman form} is given by 
		\begin{align}\label{eq:ep-1a}
		\E_p [u] := \frac{1}{p} \intRd\intRd F_p(u(x),u(x+y)) \,\nu(\dy)\dx.
		\end{align}	
	Similarly, if $(P_t)_{t\geq 0}$ is the Gaussian semigroup, then, as we verify in  Appendix~\ref{sec:brown},
	\begin{align}
		\label{eq:HSgauss2}
		\intRd |f(x)|^p \,\dx = p(p-1) \int_0^\infty \intRd |P_t f(x)|^{p-2} |\nabla P_t f(x)|^2 \,\dx\dt, 
	\end{align}

	In Theorem~\ref{thm:HS15} below we prove a variant of \eqref{eq:HS-BBL}
	for vector-valued functions\linebreak
 $F\!:\mathbb R^d\to\mathbb R^n$ with  energy $\int_\Rd|F|^p$ and multidimensional Bregman divergence $\mathcal F_p$. We use a novel approach based on differential calculus in $L^p$, but to reduce technicalities, we restrict our attention to L\'evy measures $\nu(\dx)=\nu(x)\dx$, whose density $\nu(x)$ and semigroup $P_t$ satisfy certain relative bounds \eqref{eq:mp} and \eqref{eq:zp}. 
	
	Then, in Theorem~\ref{thm:polarized}, we provide a \textit{polarized Hardy\nobreakdash--Stein identity} for
	$2 \leq p < \infty$ and the \textit{mutual energy} of two scalar functions $f, g \in \Lp$ as follows:
	\begin{equation}\label{eq.HScd}
		\intRd f(x) g(x)|g(x)|^{p - 2}\,\dx=
		p\int_0^\infty \E_p(P_t f,P_t g)\,\dt.
	\end{equation}
	Here, for suitable functions $u, v\!:\Rd\to \R$, we define 
	\begin{equation}\label{eq:polarized}
		\E_p(u,v) := \frac{1}{p} \int_{\mathbb R^d}\int_{\mathbb R^d} \mathcal J_p\big(u(x),v(x);u(y),v(y)\big)\nu(x,y)\,\dx\dy,
	\end{equation}
	which we call the \textit{polarized Sobolev-Bregman form};
	the function $\mathcal J_p$  in \eqref{eq:polarized} will be called the \textit{Bregman co-divergence}, see Section~\ref{sec:polarized}.
	Compared to \eqref{eq:HS-BBL}, the polarized Hardy--Stein identity is delicate because its integrands
	are signed and absolute integrability of the right-hand side of \eqref{eq.HScd} is in question, especially when $2 < p < 3$.

	We see in \eqref{eq:HS} that $p\mathcal E_p[P_tf]$ indeed \textit{disintegrates} the  energy of $f$. 
The formula is closely related to the fact that for nice functions $u$ and Markovian semigroups $(P_t)$ with generator $L$ we have $\tfrac{d}{dt}\int |P_t f|^p=\int p P_tf |P_tf|^{p-2}LP_tf=-p\E_p[P_tf]$, see, e.g., the discussions in Bogdan, Jakubowski, Lenczewska, and Pietruska-Pa\l{}uba \cite{MR4372148} and Bogdan, Grzywny, Pietruska-Pa\l{}uba, and Rutkowski \cite{MR4088505}, or Proposition~\ref{thm:EpDpL} below. 
The connection is known at least since Varopoulos \cite[(1.1)]{MR0803094}; see also Davies \cite[Chapter 2 and 3]{MR1103113}, Langer and Maz'ya \cite{MR1694522}, and Liskevich and Semenov \cite{MR1409835}. Accordingly, as recently highlighted in \cite{MR4372148},
Sobolev--Bregman forms are an efficient tool for estimating Markovian semigroups on $L^p$ spaces.
	For related analysis, we refer to Farkas, Jacob and Schilling \cite[(2.4)]{MR1808433} and to Jacob \cite[(4.294)]{MR1873235}.  
In the same vein, the Bregman divergence defines the energy in Gagliardo\nobreakdash--Nirenberg\nobreakdash--Sobolev inequalities, see Carrillo, J\"{u}ngel, Markowich, Toscani, and Unterreiter \cite[p. 71]{MR1853037}, Bonforte, Dolbeault, Nazaret, and Simonov \cite{2020arXiv200703674B},  and Borowski and Chlebicka \cite{MR4518648}. 
	For references to further applications of Bregman divergence in analysis, statistical learning, and optimization, we refer to Bogdan, Grzywny, Pietruska-Pa{\l}uba, and Rutkowski \cite{MR4589708}; see also Bogdan and Wi\c{e}cek \cite{MR4521651} for a martingale connection.

	Applications of Hardy\nobreakdash--Stein identities include characterizations of Hardy spaces in Bogdan, Dyda, and Luks \cite{MR3251822},
	the Littlewood\nobreakdash--Paley theory and estimates of Fourier multipliers 
	\cite{MR3556449},
	Douglas identities, extension and trace theorems, Dirichlet-to-Neumann maps, and quasiminimizers; see  \cite{MR4088505} and Bogdan, Fafu\l{}a, and Rutkowski \cite{MR4600287}. 	
	On the other hand, to the best of our knowledge, the Bregman co-divergence $\mathcal J_p$, the polarized Sobolev-Bregman form $\E_p(\cdot,\cdot)$, and the polarized Hardy\nobreakdash--Stein identity \eqref{eq.HScd} are new. Apart from being interesting for their own sake, they should be useful for studying the variational problems for $\E_p[\cdot]$.  Let us also note that the convexity of the function $|a|^p$ plays an important role in \eqref{eq:HS-BBL}, but \eqref{eq.HScd} is achieved by writing $a b|b|^{p-2}$ as a difference of two convex functions.

	The structure of the paper is as follows.  The Hardy\nobreakdash--Stein identity \eqref{eq:HSgauss2} for the Gaussian semigroup is verified in Appendix~\ref{sec:brown}.  The result is apparently new, but \cite[Section 4]{MR3251822} and \cite{MR4600287} discuss elliptic analogues. Theorem~\ref{thm:HS15} is proved in Section~\ref{sec:2func}.
Theorem~\ref{thm:polarized}, including \eqref{eq.HScd}, is proved in Section \ref{sec:polarized}.  In Section \ref{sec:form-polarized} we discuss the polarized Sobolev-Bregman form. 
	In Appendix~\ref{sec:App-A} we give auxiliary estimates for Bregman divergence and related quantities.
	Appendix~\ref{sec:App-B} presents elements of calculus in $L^p$, in particular
	Lemma~\ref{lem:otco} gives the derivative of $\int_\Rd P_tf P_tg |P_tg|^{p - 2}\,\dx$. In Appendix~\ref{sec:apx} we discuss convexity properties related to the Bregman co-divergence $\mathcal J_p$. In Appendix~\ref{s.ap} we give a simpler proof of Theorem~\ref{thm:polarized}, but only for $p\ge 3$.
	In Appendix~\ref{sec:app}, complementing Appendix~\ref{sec:brown}, we discuss the $L^p$ generator of the Gaussian semigroup. 
	For technical reasons, our main results are restricted to a class of 
	convolution Markovian semigroups on $\Rd$, but some arguments are presented in a more general setting and further extensions are forthcoming.
	For instance, Gutowski \cite{MR4622410} and Gutowski and Kwaśnicki \cite{gutowski2023beurlingdeny} extend \eqref{eq:HS-BBL}
	to general symmetric Markovian semigroups with nonlocal Dirichlet forms; see also Bogdan, Kutek, and Pietruska-Pałuba \cite{2024+KB-DK-KPP} for a recent probabilistic approach, based on stochastic integrals.
	
 Let us also note that in the case of non-symmetric $\nu$, the identity \eqref{eq:HS-BBL}  was obtained by  Ba\~{n}uelos and Kim \cite{MR3994925}. While our proof
 of \eqref{eq:polarized} heavily depends on the symmetry of $\nu$, we hope that methods similar to those from \cite{MR3994925} or the semimartingale connection from \cite{2024+KB-DK-KPP} may allow for extensions to the non-symmetric setting.
	\medskip
	
	{\bf Acknowledgments.} We thank Włodzimierz B\c{a}k, Bartłomiej Dyda, Mateusz Kwa\'{s}nicki,  Agnieszka Ka\l{}amajska, and Bart\l{}omiej Wr\'{o}bel for helpful discussions.

	\section{Preliminaries}
	\label{sec:int}
	
	\subsection{Standing assumptions}\label{ss.a}	
	\smallskip
	
	All the considered sets, functions, and measures are assumed Borel. For nonnegative functions $f$ and $g$, we write $f(x) \asymp g(x)$ to indicate that there exist \textit{constants}, i.e., numbers $0 < c \leq C < \infty$ such that $cf(x) \le g(x) \le Cf(x)$ for all the considered arguments $x$. Without warning, the symbols $c$ or $C$ may denote different constants
	even within a single line of text. The symbol $:=$ means definition, e.g., $a\vee b:=\max\{a,b\}$, $a\wedge b:=\min\{a,b\}$,
	$a_+:=a\vee0$, and
	$a_-:=(-a)\vee0$.
	We denote the Euclidean norm of a vector $z\in\R^n$ as $|z|$ and the standard scalar product of vectors $w$ and $z$ in $\R^n$ as $(w,z)$ or $w\cdot z$.
	The unit sphere in $\Rn$ centered at the origin is denoted by $\Sn$. As usual, $\|f\|_{L^q(\Rd)}$ denotes the $L^q(\Rd)$ norm of the (extended real-valued) function $f$, $1\le q\le \infty$. More specifically, 	$\|f\|_{L^q(\Rd)} :=  \big(\int_{\Rd} |f(x)|^q \,\dx)^{1/q}$ for $1\le q< \infty$, where ${\rm d}x$ refers to the Lebesgue measure on $\Rd$ and $\|f\|_{L^\infty(\Rd)} := \esssup |f|$.
	
	Let $d=1, 2, \ldots$.
	Consider a symmetric, absolutely continuous L\'evy measure $\nu$  on the Euclidean space $\Rd$. Thus, $\nu(\dz)=\nu(z)\,\dz$, where $\nu\!:\R^d\setminus\{0\}\to(0,\infty)$,   $\nu(-z)=\nu(z)$, $z\in \mathbb R^d\setminus\{0\}$, and
	\[ \int_{\Rd}\left(|z|^2\wedge 1\right)\nu(z)\,\dz <\infty.\]
	The corresponding L\'evy\nobreakdash--Khinchine exponent is
	\begin{equation}
		\label{eq:Lexponent}
		\psi(\xi):= \intRd \left(1-\cos(\xi\cdot x)\right)\nu(x)\,\dx,\quad \xi\in \Rd.
	\end{equation}
	We further assume the following Hartman\nobreakdash--Wintner condition on $\psi$ (and $\nu$):
	\begin{equation}
		\label{eq:HaWi}
		\lim_{|\xi|\to \infty}
		\frac{\psi(\xi)}{\log|\xi|}=\infty.
	\end{equation}
	In particular, $\int_{\Rd}\nu(z)\,{\rm d}z=\infty$.  
	This gives rise to a convolution semigroup of probability densities $p_t$ given by the L\'evy\nobreakdash--Khintchine formula, or Fourier inversion, as follows:
	\begin{equation}
		\label{eq:Fi}
		p_t(x):=(2\pi)^{-d}\intRd e^{-i\xi\cdot x}e^{-t\psi(\xi)}\,{\rm d}\xi,\quad t>0, \ x\in \Rd.
	\end{equation}
	The function $p_t(x)$ is continuous and attains its maximum at $x=0$, which is
	\begin{align*}
		p_t(0) & = (2\pi)^{-d}\intRd e^{-t \psi(\xi)} \,{\rm d}\xi,\quad t>0.
	\end{align*}
	By \eqref{eq:HaWi}, $p_t(0)$ is finite for every $t>0$ and, by the Dominated Convergence Theorem, $p_t(0)$ converges to zero as $t \to \infty$, so $\|p_t\|_{L^\infty(\Rd)}\to 0$.

	We shall also write $p_t(x,y):=p_t(y-x)$ and $\nu(x,y):=\nu(y-x)$. Note that $p_t(x,y)$ is a transition density of a pure-jump L\'evy stochastic process $\{ X_t, t\geq0\}$ in $\Rd$ with the L\'evy\nobreakdash--Khintchine exponent $\psi$ (see Sato \cite{MR1739520}) and $\nu(x,y)$ is the kernel of the corresponding Dirichlet form; see, e.g., Fukushima, Oshima, and Takeda \cite{MR2778606} and \eqref{eq:E2}. (In the following discussion, the process does not play a significant role.)
	Encouraged by \cite{MR4372148}, we also assume the following technical conditions:
	\begin{equation}
		\phantom{aaa}\tag{\bf{P1}}\label{eq:mp}
		p_t(x,y)/t \leq c \nu(x,y), \quad t>0,\ x,y\in\Rd,
	\end{equation}
	with some constant $c$, and
	\begin{equation}
		\phantom{aaa}\tag{\bf{P2}}\label{eq:zp}
		p_t(x,y)/t \rightarrow \nu(x,y) \mbox{ as } t \rightarrow 0^+,\ x,y\in\Rd.
	\end{equation}
For instance, the transition density corresponding to the fractional Laplacian satisfies \eqref{eq:mp} and \eqref{eq:zp}; see  examples provided by Bogdan, Grzywny, and Ryznar \cite[Corollary 23]{MR3165234} and Cygan, Grzywny, and Trojan \cite[Proof of Theorem 6]{MR3646773}. The conditions are convenient in limiting procedures based on the Dominated Convergence Theorem. They can probably be relaxed, as in the (nonpolarized) Hardy--Stein identities given by \cite{2024+KB-DK-KPP}, \cite{MR4622410},  \cite{gutowski2023beurlingdeny}. See also Appendix~\ref{sec:brown} below, where we give the (nonpolarized) Hardy--Stein identity for the Brownian motion. In this case, $p_t(x,y)/t\to 0$ when $t\to 0^+$ and $x\neq y$.
Based on \cite{gutowski2023beurlingdeny}, however, we anticipate that relaxing \eqref{eq:mp} and \eqref{eq:zp} will significantly complicate the arguments.

	\subsection{Elementary functions and inequalities}
	
	Throughout we use the notation
	\begin{align*}
		a^{\langle \kappa \rangle} := \left|a \right|^\kappa \sgn 	 a= a|a|^{\kappa-1},\quad a, \kappa \in \R,
	\end{align*}
	where $0^{\langle \kappa \rangle}:=0$ and, as usual, $\sgn 0=0$,
	$0^0:=1$, $0^{\kappa}:=\infty$ for $\kappa<0$, and $0\cdot\infty:=0$.
	Note that
	\begin{align*}
		\left(
		|x|^\kappa
		\right)' = \kappa x^{\langle \kappa-1 \rangle}
		\quad
		\mbox{if }
		x\in\R\mbox{ and }\kappa>1
		\mbox{ or }
		x\in\R\!\setminus\!\{0\}\mbox{ and }\kappa\in\R.
	\end{align*}
	Furthermore,
	\begin{align*}
		\left(
		x^{\langle \kappa \rangle}
		\right)' = \kappa |x|^{\kappa-1 }
		\quad
		\mbox{if }
		x\in\R\mbox{ and }\kappa>1
		\mbox{ or }
		x\in\R\!\setminus\!\{0\}\mbox{ and }\kappa\in\R.
	\end{align*}
	This has a vector counterpart: for $\kappa>0$, we let
	\begin{equation}\label{eq:z-kappa-vect}
		z^{\langle\kappa\rangle}:= |z|^{\kappa-1}z,\quad z\in\Rn,
	\end{equation}
	again with the convention $0^{\langle \kappa \rangle}:=0$.
	Note that
	\begin{align}
		\label{eq:Nabla|z|^p}
		\nabla|z|^\kappa = \kappa z^{\langle\kappa-1\rangle}
		\quad
		\mbox{if }
		z\in\Rn \mbox{ and } \kappa>1
		\mbox{ or }
		z\in\Rn\!\setminus\!\{0\} \mbox{ and }  \kappa\in\R.
	\end{align}
	Furthermore, the Jacobi matrix $J_{\langle\kappa\rangle}$ for the mapping 
	$z\mapsto z^{\langle \kappa \rangle}$
	equals
	\begin{equation}\label{eq:diff-french-n}
		J_{\langle\kappa\rangle}(z) = |z|^{\kappa-1}\left((\kappa-1)\left( \frac{z}{|z|}\otimes \frac{z}{|z|}\right) + \mbox{Id}\right) \in \Rn\times \Rn\quad\mbox{ if }z\in\Rn\!\setminus\!\{0\}
	\end{equation}
	and we let $J_{\langle\kappa\rangle}(0):=0$.
	
	In the following, unless otherwise specified, we consider exponents $p\in (1,\infty)$.
	
	\begin{defn}\label{def:bregman}
		The \emph{Bregman divergence}
		$\mathcal F_p\!:\R^n\times\R^n\to(0,\infty)$ is given  by
		\begin{align}
			\label{eq:Fp}
			\mathcal F_p(w,z)
			:=
			|z|^p-|w|^p - pw^{\langle p - 1 \rangle}\cdot(z-w)
		\end{align}
		and the \emph{symmetrized Bregman divergence} is
		\begin{align}
			\label{eq:Hp}
			\mathcal H_p(w,z)
			:=
			\frac{1}{2}
			\left(
			\mathcal F_p(w,z)+\mathcal F_p(z,w)
			\right)
			=
			\frac{p}{2}
			(z-w)\cdot
			\left(
			z^{\langle p - 1 \rangle} - w^{\langle p - 1 \rangle}
			\right).
		\end{align}
	\end{defn}
	For instance,
	\begin{align}
		\label{eq:F2}
		\mathcal F_2(w,z)=\mathcal H_2(w,z)=|z-w|^2,\quad w,z\in \R^n.
	\end{align}
	Note that $\mathcal F_p(w,z)=|z|^p$ if $w=0$, but $\mathcal F_p(w,z)=(p-1)|w|^p$ if $z=0$.
	Of course, the mapping $\R^n\ni z\mapsto |z|^p$ is convex, since $p>1$. Its second-order Taylor remainder is $\mathcal F_p$, so $\mathcal F_p\ge 0$ and $\mathcal H_p\ge 0$.
	Also, if $Q$ is an $n\times n$ orthogonal matrix, then
	\begin{align}
		\label{eq:Fpsym}
		\mathcal{F}_p(Qw,Qz) = \mathcal{F}_p(w,z).
	\end{align}
	For notational convenience in what follows, we also let
	\begin{equation}\label{eq:gp-vect}
		\mathcal G_p(w,z) := |z-w|^2 (|w|\vee |z|)^{p-2}, \quad z,w\in\Rn.
	\end{equation}
	We further introduce the second-order Taylor remainders of the vector functions \linebreak $\!\R^n\!\ni\! z\mapsto z^{\langle \kappa \rangle}\in\R^n$ for  $\kappa >1$.  \!\!
	More precisely, we define $\mathcal F_{\langle\kappa\rangle}\!:\Rn\times\Rn\to \R^n$ by
	\[\mathcal F_{\langle \kappa \rangle}(w,z):=z^{\langle \kappa\rangle}- w^{\langle \kappa \rangle} - J_{\langle\kappa\rangle}(w)(z-w),\;\quad w,z\in\mathbb R^n.\]
	Of course, the mapping $\R\ni x\mapsto x^{\langle\kappa\rangle} \in\R$ is in general not convex and $\mathcal F_{\langle \kappa \rangle}$ changes sign. In fact,
	$\mathcal F_{\langle\kappa\rangle}(-a,-b) = -{\mathcal F}_{\langle\kappa\rangle}(a,b)$.
	
	The scalar versions of the above functions (for $n=1$)  are denoted $F_\kappa$ (see Introduction), $H_\kappa$, $G_\kappa$, and $F_{\langle\kappa\rangle}$, 
respectively.
	In particular, 
	\[F_{\langle\kappa\rangle}(a,b)= b^{\langle \kappa\rangle}- a^{\langle \kappa\rangle}-
	\kappa |a|^{\kappa-1}(b-a),\quad a,b\in\R.\]
	The following estimates are quite important for our analysis.
	\begin{lem}
		\label{lem:Fp2Gp2}
		Let $p\in(1,\infty)$. We have
		\begin{align}
			\label{eq:Fp2Gp2}
			\mathcal{F}_p(w,z) \asymp \mathcal{G}_p(w,z),\quad w,z\in\R^n,
		\end{align}
		and
		\begin{align}
			\label{eq:Hp2Gp2}
			\mathcal{H}_p(w,z) \asymp \mathcal{G}_p(w,z),\quad w,z\in\R^n.
		\end{align}
	\end{lem}
	Of, course \eqref{eq:Hp2Gp2} follows from \eqref{eq:Fp2Gp2}. 
	It seems that \eqref{eq:Fp2Gp2}	was first proved in Pinchover, Tertikas, and Tintarev \cite[(2.19)]{MR2400106}, but one of the  one-sided bounds was given earlier in Shafrir \cite[Lemma 7.4]{MR1759788} for $p\ge 2$ and the other in Barbatis, Filippas, and Tertikas \cite[Lemma~3.1]{MR2048514}. The one-dimensional case, $F_p(a,b) \asymp (b-a)^2(|a|\vee |b|)^{p-2}$, $a,b\in \R$, is crucial in \cite[Lemma 6]{MR3251822}, with \cite[(10), (12)]{MR3251822} therein being a part of the comparison for $n=2$ (the setting of \eqref{eq:Fp2Gp2} is essentially two-dimensional). Optimal constants are known in some cases: for the lower bound of $F_p$ with $p\in (1,2)$ and for the upper bound with $p\in (2,\infty)$; see \cite{MR4372148} and \cite[Lemma 7.4]{MR1759788}.
	The quadratic factor in \eqref{eq:gp-vect} is the reason why Bregman divergence $\mathcal{F}_p$ is integrable against L\'evy measures in \eqref{eq:ep-1a}, which is crucial in analysis of nonlocal equations of parabolic and elliptic type; see the applications of \cite[(2.14)]{MR4589708} therein. See also \cite{MR4372148} for a martingale setting.  In passing we note another important estimate:
	\begin{equation}\label{e.cHk}
		\mathcal F_p(w,z)\asymp \left|z^{\langle p/2\rangle}-w^{\langle p/2\rangle}\right|^2, \quad w,z\in \R^n.
	\end{equation}
	We refer to \cite[Subsection 1.3]{MR4372148} for a discussion of the estimate when $n=1$; the case of arbitrary $n= 1,2,\ldots$ can be found in Huang \cite{MR1923834}.
	Further estimates concerning functions $\mathcal F_p$, $\mathcal F_{\langle \kappa\rangle}$, and their cousins 
are collected and proved in  Appendix~\ref{sec:App-A}.

	\subsection{Semigroups, generators, and forms}
	The semigroup is defined by 
	\begin{align*}
		P_t f(x):=\intRd f(y) p_t(x,y) \,\dy,\quad t>0,
	\end{align*}
	and by $P_0 f(x) := f(x)$, where $x\in\Rd$ and $f\!:\Rd\to \R$ is nonnegative or integrable. We briefly mention a well known probability connection:
	$P_t f(x)=\Ex f(X_t)$, where $(X_t,\mathbb P_x)_{t\geq 0,x\in \mathbb R^d}$ is our L\'{e}vy process, considered as a Markov process  with transition density on $\Rd$ given by $p_t(\cdot,\cdot)$, and $\mathbb E_x$ is the expectation with respect to the distribution $\mathbb P_x$ of the process starting from $x$.
	
	Since $\int_{\mathbb R^d} p_t(x,y)\,\dy =1$ for $t>0$,  $x\in \mathbb R^d$ (conservativeness of the semigroup $P_t$),  
and $p_t$ is symmetric  in $x,y,$ Fubini-Tonelli theorem yields
	\begin{align}\label{eq:intPtf}
		\intRd P_t f(x)\,\dx = \intRd f(x)\,\dx.
	\end{align}

	Recall that $1<p<\infty$. It is well known that $(P_t)_{t\geq 0}$ is a strongly continuous Markovian semigroup of symmetric operators on $\Lp$; see for example \cite[E 34.10]{MR1739520}. For all $x\in\Rd$ and $f\in\Lp$, by \eqref{eq:Fi} and H\"{o}lder's inequality with exponents $p$ and $q=p/(p-1)$, we get
	\begin{align}
		\label{eq:Ptf}
		\left| P_t f (x) \right|
		&=
		\left|\intRd f(y)p_t(x,y)\,\dy\right|
		\leq
		\|f\|_{\Lp}\left(\intRd p_t(x,y)^q \,\dy\right)^{1/q}
		\\ \nonumber
		&\leq
		\|f\|_{\Lp}\left( \sup_{x,y\in\Rd} p_t(x,y)^{q-1} \right)^{1/q}
		=
		\|f\|_{\Lp} \,\|p_t\|_{L^\infty(\Rd)}^{1/p}
		\xrightarrow[t\to\infty]{}
		0.
	\end{align}

	We also need the following maximal inequality of Stein for symmetric Markovian semigroups; see Stein \cite[p. 73]{MR0252961} and recall that $1<p<\infty$.
	\begin{lem}
		[Stein inequality]
		\label{lem:stein}
		If $f\in \Lp$, 
		$f^*(x):=\sup_{t\ge0}|P_tf(x)|$, $x\in \Rd$,
		then,
		\begin{equation}
			\label{eq:stein}
			\|f^{*}\|_{\Lp}\leq \frac{p}{p-1}\|f\|_{\Lp}.
		\end{equation}
	\end{lem}
	By \eqref{eq:stein} and 
	\eqref{eq:Ptf},
	the semigroup is \textit{strongly stable} in $L^p(\Rd)$: If $f\in L^p(\Rd)$, then
	\begin{equation}
		\label{Pt-strongly-stable}
		\|P_t f\|_{\Lp}\to 0 \mbox{ as } t\to\infty.
	\end{equation}
	Indeed, since for every $x\in\Rd$ we have $|P_tf(x)|\to 0$ and $|P_tf(x)|\leq f^*(x)$   with $f^*\in L^p(\Rd)$,   we get
	$\|P_tf\|_{\Lp} 
	\to 0$
	by the Dominated Convergence Theorem.

	Let $L$ be the generator of the semigroup $(P_t)_{t\geq 0}$, when acting on $\Lp$. Its natural domain, denoted $\DpL$, consists of those $f\in \Lp$ for which there is $g\in \Lp$ such that $(P_hf-f)/h\to g$ in $\Lp$ as $h\to 0^+$; we then write $Lf=g$.

	We next discuss issues related to the $L^p$-differentiability of semigroups. (To make the exposition self-contained, we include a primer on the $L^p$ calculus in Appendix~\ref{sec:App-B}.)
	Thus, for $f\in\Lp$ and $t\ge 0$, we write $u(t) := P_tf$. Of course, $u(t)\in\Lp$.
	Furthermore, if $f\in\DpL$ then $u'(t) = L P_t f = P_t L f=Lu(t)$, $t\geq0$.
	By Lemma~\ref{thm:difpow} with $n=1$,  we obtain the following result.
	\begin{cor}
		\label{lem:pdif}
		Let $f\in\DpL$.
		If $1<\kappa\le p$ then:
		\begin{enumerate}
			\item[\bf (i)]
			$|u(t)|^\kappa$ is continuously differentiable in $L^{p/\kappa}(\Rd)$ and
			\begin{equation}
				\label{eq:ppp}
				(|u(t)|^\kappa)' =\kappa u(t)^{\langle \kappa - 1 \rangle} u'(t)= \kappa u(t)^{\langle \kappa - 1 \rangle} P_t Lf, \quad t \geq 0,
			\end{equation}
			\item[\bf (ii)]  $u^{\langle \kappa \rangle}$ is continuously differentiable in $L^{p/\kappa}(\Rd)$ and
			\begin{align}
				(u(t)^{\langle \kappa \rangle})' = \kappa|u(t)|^{\kappa-1} u'(t)
				=
				\kappa|u(t)|^{\kappa-1} P_t L f,     \quad t \ge 0.
			\end{align}
		\end{enumerate}
	\end{cor}
	
	Moreover, since $(P_t)_{t\geq 0}$ is symmetric, it is an analytic semigroup on $\Lp$ for $p\in(1,\infty)$; see Liskevich and Perel'muter \cite[Corollary 3.2]{MR1224619}. Therefore, for all $t>0$ and $f\in\Lp$, 
	$\frac{\rm d}{\dt} P_t f=u'(t)$
	exists in $\Lp$, so $P_t f \in \DpL$ and
	$u'(t) = L P_t f = L u(t)$.
	
	\smallskip
	
	As a special case of \eqref{eq:ep-1a}, in what follows we consider the integral form
	\begin{align}\label{eq:ep-1}
		\E_p [u] = \frac{1}{p} \intRd\intRd F_p(u(x),u(y)) \nu(x,y)\,\dx\dy.
	\end{align} 
	Of course, the form is well-defined (possibly infinite) for every $u\!:\Rd\to\R$ because $F_p\geq 0$. By the symmetry of $\nu$,
	\begin{eqnarray}
		\label{eq:Epsym}
		\E_p [u]
		&=&
		\frac{1}{p} \intRd\intRd H_p(u(x),u(y)) \nu(x,y)\,\dx\dy \nonumber
		\\
		&=&
		\frac{1}{2} \intRd\intRd
		(u(y)-u(x))
		\left(
		u(y)^{\langle p - 1 \rangle} - u(x)^{\langle p - 1 \rangle}
		\right)
		\nu(x,y)\,\dx\dy.
	\end{eqnarray}
	The natural domain of $\E_p$ is
	\begin{align}
		\DEp := \{u\in\Lp:\ \E_p[u]<\infty\}.
	\end{align}
	When $p=2$,  we get the usual Dirichlet form of the semigroup,
	\begin{align}
		\label{eq:E2}
		\mathcal E_2[u]= \frac{1}{2} \intRd\intRd
		(u(y)-u(x))^2
		\nu(x,y)\,\dx\dy,
	\end{align}
	with domain $\mathcal D(\mathcal E_2)$. We write $\E:=\E_2$.
	
	For $t>0$, $u\in\Lp$, and $v\in L^q(\Rd)$, we define, as usual,
	\begin{align}\label{eq:Etdef}
		\E^{(t)}(u,v)
		:=
		\frac{1}{t}
		\langle u-P_tu,v \rangle =\frac{1}{t}\int_{\Rd} (u(x)-P_tu(x))v(x)\,\dx.
	\end{align}
	Here and below we use the following notation for the canonical paring:
	\begin{align*}
		\langle u,v \rangle := \intRd u(x)v(x)\,\dx.
	\end{align*}
	The next result was established in \cite[Lemma 7]{MR4372148} for the fractional Laplacian, but since its proof requires only symmetry and the conditions \eqref{eq:mp} and \eqref{eq:zp}, it applies verbatim in the present setting.
	
	\begin{prop}
		\label{thm:EpDpL}
		Let $p>1$. For every $u\in \Lp$, we have
		\begin{equation}\label{eq:wpf}
			\E_p[u]= \lim_{t\to 0} \E^{(t)}(u,u^{\langle p-1 \rangle}).
		\end{equation}
		Furthermore,
		\begin{eqnarray}\label{eq:dom_def-1}
			\mathcal D(\E_p)
			& = &
			\{u\in \Lp: \sup_{t>0}\mathcal E^{(t)}(u,u^{\langle p-1 \rangle})<\infty\}\\
			\label{eq:dom_def-2}
			&=&
			\{u\in \Lp: \text{ finite } \lim_{t\to 0}\mathcal E^{(t)}(u,u^{\langle p-1 \rangle}) \text{ exists}\}.
		\end{eqnarray}
		For arbitrary $u\!: \Rd \to \R$, we have
		\begin{equation}\label{eq:cHkf}
			\frac{4(p-1)}{p^2}\E[u^{\langle p/2\rangle}]\le \E_p[u]\le 2 \E[u^{\langle p/2\rangle}]
		\end{equation}
		and $\mathcal D(\E_p)=\mathcal D(\E)^{\langle 2/p\rangle}:=\{v^{\langle 2/p\rangle}: v\in \mathcal D(\E)\}$.
		Finally, $\mathcal D_p(L)\subset  \mathcal D(\E_p)$ and
		\begin{equation}\label{eq:pfag}
			\E_p[u]=-\langle Lu, u^{\langle p-1\rangle}\rangle, \quad u\in \mathcal D_p(L).
		\end{equation}
	\end{prop}

	The discussion extends to functions with values in $\R^n$, $n=1,2,\ldots$.  
	Namely, let $f_1,\ldots,f_n\in L^p(\Rd)$, so $F:=(f_1,\ldots,f_n)\in L^p(\Rd;\R^n)$. We denote
	\begin{equation}\label{eq:defpt-n}
		P_tF:=(P_tf_1,\ldots,P_tf_n), \quad t\ge 0,
	\end{equation}
	thus $P_tF\in L^p(\Rd;\R^n)$, $t\ge 0$.
	If, furthermore, $f_1,\ldots,f_n\in \mathcal D_p(L)$, then we define
	\begin{equation}\label{eq:Lndef}
		LF:=(Lf_1,\ldots,Lf_n).
	\end{equation}
	Then, letting $U(t):=P_t F(t)$, $t\ge 0$, we get
	$U'(t)= LU(t)$ and the following multidimensional extension of Corollary~\ref{lem:pdif}, an easy consequence of Lemma~\ref{thm:difpow}.
	\begin{cor}\label{lem:pdiff-n}
		Let $n=1,2,\ldots$, $f_1,\ldots,f_n\in\DpL$, $F:=(f_1,\ldots,f_n)$,
		and $U(t)=P_t F$, $t\ge 0$.
		If $1<\kappa\leq p$, then for $t\ge 0$,
		\begin{enumerate}
			\item[\bf(i)] $|U(t)|^\kappa$ is continuously differentiable in $L^{p/\kappa}(\Rd)$ with
			\[\left(|U(t)|^\kappa\right)'=\kappa U(t)^{\langle\kappa -1\rangle}\cdot LU(t),
			\]
			\item[\bf(ii)] $U(t)^{\langle\kappa\rangle}$ is continuously differentiable in $L^{p/\kappa}(\Rd;\R^n)$ with
			\[\left( U(t)^{\langle \kappa\rangle}\right)'= \left(J_{\langle\kappa\rangle}\circ U(t)\right) LU(t).\]
		\end{enumerate}
	\end{cor}

	\smallskip
	
	The following result will be useful in limiting procedures later on.
	\begin{lem}{\cite[Lemma 6]{MR4372148}}
		\label{lem:Fatou}
		If nonnegative functions $f,f_k\!:\Rd\to\R$, $k=1,2,\ldots$,  satisfy
		$f_k\le c f$ and $f=\lim_{k\to\infty} f_k$, then
		$\lim_{k\to\infty} \int f_k \,{\rm d}\mu= \int f \,{\rm d}\mu$ for each measure $\mu$.
	\end{lem}

	\section{Hardy--Stein identity}
	\label{sec:2func}
	Below we work under the assumptions on $\nu$ formulated in Subsection~\ref{ss.a}.
	We will extend \eqref{eq:HS-BBL} to arbitrary dimension $n=1,2,\ldots$. We recall that the proof given in \linebreak
 \cite[Theorem~3.2]{MR3556449} for  $n=1$ relies on approximations and pointwise calculus in $\Rd$. Here, instead, we use a more synthetic differential calculus in $L^p$. 
	\begin{thm}
		\label{thm:HS15}
		Let $p>1$, $n=1,2,\ldots$,  and $F=(f_1,\ldots,f_n)\in L^p(\Rd;\R^n)$.   Then,
		\begin{align}
			\label{eq:HS15}
			\intRd |F(x)|^p\,\dx
			=
			\int_0^\infty \intRd \intRd
			\mathcal{F}_p (P_tF(x),P_tF(y)) \nu(x,y)
			\,\dx\dy\dt.
		\end{align}
	\end{thm}
	\begin{proof}
		Let first $F=(f_1,\ldots,f_n)\in(\DpL)^n$ and $0\leq t\le T<\infty$. Then $U(t):=P_tF\in (\DpL)^n$
		and $LP_tF= LU(t)= (LP_tf_1,\ldots,LP_tf_n)$.
		From Corollary~\ref{lem:pdiff-n}, $|U(t)|^p$ is continuously differentiable in $L^1(\Rd)$ and 
		$\left( |U(t)|^p \right)'
		=
		pU(t)^{\langle p-1 \rangle}\cdot  L F(t)$.
		As $f\mapsto\intRd f\,\dx$ is a continuous linear functional on $L^1(\Rd)$,
		\begin{eqnarray}
			\label{eq:ddtintmod}
			\frac{\rm d}{\dt}  \intRd |U(t)|^p\,\dx
			&=&
			\intRd \frac{\rm d}{\dt}|U(t)|^p\,\dx
			=
			\intRd p U(t)^{\langle p-1 \rangle}\cdot  L U(t) \,\dx
			\nonumber
\\
			&=&
			\langle  L U(t), p U(t)^{\langle p-1 \rangle} \rangle.
		\end{eqnarray}
		Since  $ LU(t) = \lim_{h\to 0^+} (P_h U(t)-U(t))/h$  strongly in $L^p(\Rd;\R^n)$,     $U(t)^{\langle p-1\rangle}$ belongs to the (dual) space $L^{\frac{p}{p-1}}(\Rd;\R^n)$, and the semigroup $(P_t)_{t\geq 0}$ is conservative, we get
		\begin{eqnarray}
			\nonumber
			&&\langle  L U(t), p U(t)^{\langle p-1 \rangle} \rangle 
			\\
			\nonumber
			&=&
			\lim_{h\to0^+}
			\intRd \intRd
			pU(t)(x)^{\langle p-1 \rangle}\cdot(U(t)(y) - U(t)(x))
			\frac{p_h(x,y)}{h}
			\,\dx\dy
			\\
			\nonumber
			&=&
			\lim_{h\to0^+}\left[
			\intRd \intRd
			pP_tF(x)^{\langle p-1 \rangle}\cdot(P_tF(y) - P_tF(x))
			\frac{p_h(x,y)}{h}
			\,\dx\dy \right.\\
			&\quad &+
			\left.
			\frac{1}{h} \intRd |P_tF(x)|^p \,\dx
			-
			\frac{1}{h} \intRd P_h(|P_tF|^p)(x) \,\dx
			\right]
			\nonumber \\
			&=&
			-\lim_{h\to0^+} \intRd \intRd
			\mathcal{F}_p (P_tF(x),P_tF(y))
			\frac{p_h(x,y)}{h}
			\,\dx\dy
			\nonumber \\
			&=&
			- \intRd \intRd
			\mathcal{F}_p (P_tF(x), P_tF(y))
			\nu(x,y)
			\,\dx\dy.
			\label{eq:diff15}
		\end{eqnarray}
		The last equality  \eqref{eq:diff15}  is justified by
		Lemma~\ref{lem:Fatou}, the nonnegativity of $\mathcal{F}_p$, and
		assumptions \eqref{eq:mp}, \eqref{eq:zp}.
		Summarizing, 
		\[\frac{\rm d}{\dt}  \intRd |U(t)|^p \dx
		=
		- \intRd \intRd
		\mathcal{F}_p (P_tF(x),P_tF(y))
		\nu(x,y)
		\,\dx\dy.\]
		Since $|U(t)|^p$ is continuously differentiable in $L^1(\Rd)$ on ${[0,\infty)}$
		and the integration is a continuous linear functional on $L^1(\Rd)$, 
		$\intRd |U(t)|^p\,\dx$ is continuously differentiable.
		Integrating from $0$ to $T$, we obtain
		\begin{eqnarray*}
			\intRd |F|^p\,\dx - \intRd |U(T)|^p\,\dx
			&=&
			-\int_0^T \left(\frac{\rm d}{\dt} \intRd |U(t)|^p \,\dx \right)\,\dt
			\nonumber \\
			&=&
			\int_0^T \intRd \intRd
			\mathcal{F}_p (P_tF(x),P_tF(y))
			\nu(x,y)
			\,\dx\dy\dt.
		\end{eqnarray*}
		We let $T\to\infty$ and obtain $\intRd |U(T)|^p\,\dx \to 0$ from the strong stability \eqref{Pt-strongly-stable}.
		
		We now  relax the assumption $f_j\in\DpL$. Let $F=(f_1,\ldots,f_n)\in L^p(\Rd;\R^n)$ be arbitrary and  let $s>0$. Since $(P_t)_{t \geq 0}$ is an analytic semigroup on $\Lp$, $P_s f_j\in\DpL$ for all $j=1,\ldots,n$, so $U(s)\in(\DpL)^n$. By \eqref{eq:HS15} and a change of variables,
		\begin{align}
			\label{eq:HS15s}
			\intRd |U(s)|^p\,\dx =
			\int_s^\infty \intRd \intRd
			\mathcal{F}_p (P_tF(x),P_tF(y))
			\nu(x,y)
			\,\dx\dy\dt.
		\end{align}
		Let $s$ decrease to $0$.
		Since $\mathcal{F}_p\ge0$,  the right-hand side of \eqref{eq:HS15s} increases to
		$\int_0^\infty \intRd \intRd \mathcal{F}_p (P_tF(x),P_tF(y)) \nu(x,y)\,\dx\dy\dt$.
		By the strong continuity of $(P_t)_{t \geq 0}$ in $\Lp$, $P_sf_j\to f_j$,  $j=1,\ldots,n$,  in $\Lp$, so $U(s)=P_sF\to F$ in $L^p(\Rd;\R^n)$, in particular $\|U(s)\|^p_{L^p(\Rd;\R^n)}\to \|F\|^p_{L^p(\Rd;\R^n)}$.
		The proof is complete.
	\end{proof}
	
	\begin{rem}
		Since $\nu$ is symmetric,  by \eqref{eq:Hp} we get a symmetrized version of the Hardy\nobreakdash--Stein identity for every $F\in L^p(\Rd;\R^n)$:
		\begin{align*}
			\int\limits_\Rd |F|^p \,\dx
			&=
			\frac{p}{2} \int\limits_0^\infty\!\int\limits_\Rd\!\int\limits_\Rd
			(P_tF(y)-P_tF(x))\!\cdot\!
			\left(
			P_tF(y)^{\langle p - 1 \rangle} - P_tF(x)^{\langle p - 1 \rangle}
			\right)\!\nu(x,y)
			\,\dx\dy\dt.
		\end{align*}
	\end{rem}
	
	\section{Polarized Hardy--Stein identity}\label{sec:polarized}
	Having proved the Hardy\nobreakdash--Stein identity for a vector of  $\Lp$ functions, we can establish a disintegration of $\intRd f(x)g(x)^{\langle p-1 \rangle}\,\dx$
	for $f,g\in \Lp$ with $p\in[2,\infty)$.
	
	To this end we introduce the function $\mathcal J_p\!:\mathbb R^2\times
	\mathbb R^2\to\mathbb R$, defined as follows:
	\begin{align}
		\label{eq:Jp}
		\mathcal J_p(w,z)=&\mathcal J_p(w_1,w_2;z_1,z_2)
		:= z_1z_2^{\langle p-1\rangle} - w_1w_2^{\langle p-1\rangle}\\\nonumber
		&-w_2^{\langle p-1\rangle}(z_1-w_1)-(p-1)w_1|w_2|^{p-2}(z_2-w_2),
	\end{align}
	where $w=(w_1,w_2)$, $z=(z_1,z_2)$, and $w_1,w_2,z_1,z_2\in\R$. For instance,
	\begin{align}
		\label{eq:Jp2}
		\mathcal J_2(w,z)
		= z_1z_2 - w_1w_2-w_2(z_1-w_1)-w_1(z_2-w_2)
		=
		(z_1-w_1)(z_2-w_2).
	\end{align}
	As complicated as it looks, $\mathcal J_p$ is just the second-order Taylor remainder of the mapping $\mathbb R^2\ni (z_1,z_2)\mapsto z_1z_2^{\langle p-1\rangle}$, when the argument changes from $w$ to $z$. Below we mostly apply $\mathcal J_p$ to $w_1=P_tf(x)$, $w_2=P_tg(x)$, $z_1=P_tf(y)$, and $z_2=P_tg(y)$, so $w$ corresponds to the argument $x$ of the vector function $\Phi=(f,g)$, $z$ corresponds to $y$, the subscript $1$ indicates the first function, $f$, and $2$ indicates the second function, $g$.

	\smallskip
	
	Here is the main result of the paper, which we prove
below in this section.	
	\begin{thm}[Polarized Hardy\nobreakdash--Stein identity]\label{thm:polarized}
		Let $p\geq2$. For $f,g\in\Lp$,  denote
		\begin{equation}\label{eq:denote}
			\Phi(x) := (f(x),g(x))\quad{\rm and}\quad P_t\Phi(x) := (P_tf(x),P_tg(x)),\quad t\ge 0,\; x\in \Rd.
		\end{equation}
		Then,
		\begin{equation}\label{eq:int-abs}
			\int_0^\infty\int_{\mathbb R^d}\int_{\mathbb R^d}\left|\mathcal J_p(\mP_t\Phi(x),\mP_t\Phi(y))\right|\nu(x,y)\,\dx\dy\dt
			< \infty
		\end{equation}
		and
		\begin{align}
			\label{eq:HS2v2}
			\intRd fg^{\langle p - 1 \rangle}\,\dx
			=
			\int_0^\infty\int_{\mathbb R^d}\int_{\mathbb R^d}\mathcal J_p(\mP_t\Phi(x),\mP_t\Phi(y))\nu(x,y)\,\dx\dy\dt.
		\end{align}
	\end{thm}
Note that if $w_1=w_2=:a$ and $z_1=z_2=:b$, then
	$\mathcal J_p(w,z)=F_p(a,b)$, so \eqref{eq:HS2v2} with $f=g$ agrees with \eqref{eq:HS-BBL}, at least for $p\ge 2$.
	\begin{rem}\label{r.p2}
		If $p=2$  then \eqref{eq:HS2v2} reads 
		\begin{align}
			\label{eq:HS2v2bp}
			\intRd fg\,\dx
			=
			\int_0^\infty\int_{\mathbb R^d}\int_{\mathbb R^d}[P_tf(x)-P_tf(y)][P_tg(x)-P_t g(y)]\nu(x,y)\,\dx\dy\dt.
		\end{align}		
		In this case \eqref{eq:int-abs} and \eqref{eq:HS2v2} are obtained by polarization from the one-dimensional 
		Hardy\nobreakdash--Stein identity \eqref{eq:HS-BBL} and Cauchy-Schwarz inequality, 
		by considering $f+g$ and $f-g$. Therefore below we let $p>2$.
	\end{rem}

	Had $\mathcal J_p$ been nonnegative, the proof of \eqref{eq:HS2v2} would follow as that of \eqref{eq:HS15}. Unfortunately, this is not the case, so the proof is more complicated.
	Indeed, the function
$(z_1,z_2)\mapsto z_1z_2^{\langle p-1\rangle}$  
is not convex, even when restricted to $z_2>0$. To see this, we compute its gradient and Hessian matrix for $z_2>0$:
	\begin{align*}
		\nabla \left(z_1z_2^{p-1}\right)
		=
		\begin{bmatrix}
			z_2^{p-1} \\
			(p-1)z_1 z_2^{p-2}
		\end{bmatrix}
		,
	\end{align*}
	\begin{align}
		\nabla^2\left(z_1z_2^{p-1}\right)
		=
		\begin{bmatrix}
			0                &  (p-1)z_2^{p-2}          \\
			(p-1)z_2^{p-2}   &  (p-1)(p-2)z_1z_2^{p-3}
		\end{bmatrix}
		.
		\label{rem:z1z2pm1Hessian}
	\end{align}
	Thus,
	$\det \nabla^2\left(z_1z_2^{p-1}\right) = - (p-1)^2 z_2^{2p-4} < 0$,
	so
	the Hessian matrix $\nabla^2\left(z_1z_2^{p-1}\right)$ is not positive semi-definite
	and $z_1 z_2^{p-1}$ is not convex.
	We will rectify this situation by decomposing the mapping $$[0,\infty)\times\R\ni z=(z_1,z_2)\mapsto z_1z_2^{\langle p-1\rangle}$$
	into a difference of two convex mappings. Then (the  Taylor remainder) $\mathcal J_p$ will be a difference of two nonnegative  functions. To this end, we recall that $a_+:=a\vee 0,$ $a_-:=(-a)\vee 0$ and introduce the functions:
	\begin{eqnarray*}
		Y^{(+)}(z) & := & z_1 \left((z_2)_+\right)^{p-1} + |z|^p, \\
		Y^{(-)}(z) & := & z_1 \left((z_2)_-\right)^{p-1} + |z|^p,\quad z=(z_1,z_2)\in\R^2.
	\end{eqnarray*}
	Lemma~\ref{lem:Ypm} in Appendix~\ref{sec:apx} verifies that
	these functions are  convex on $[0,\infty)\times\R$ indeed.  Since $p>2$, they are differentiable everywhere and their Taylor remainders are nonnegative on $[0,\infty)\times \R$.
	
	Let
	$\mathcal{J}_p^{(+)}$ and $\mathcal{J}_p^{(-)}$ be the second-order Taylor remainders of the differentiable mappings
	$\R^2\ni z\mapsto z_1\left((z_2)_+\right)^{p-1}$ and  $\R^2\ni z\mapsto z_1\left((z_2)_-\right)^{p-1}$. Thus, for $z_1,z_2,w_1,w_2\in\R$,
	\begin{eqnarray*}
		\mathcal{J}_p^{(+)}(w,z)
		&=&
		z_1\left((z_2)_+\right)^{p-1} - w_1\left((w_2)_+\right)^{p-1}-\left((w_2)_+\right)^{p-1}(z_1-w_1)
		\\
		&&-(p-1)w_1\left((w_2)_+\right)^{p-2}(z_2-w_2)
	\end{eqnarray*}
	and
	\begin{eqnarray*}
		\mathcal{J}_p^{(-)}(w,z)
		&=&
		z_1\left((z_2)_-\right)^{p-1} - w_1\left((w_2)_-\right)^{p-1}-\left((w_2)_-\right)^{p-1}(z_1-w_1)
		\\
		&&+(p-1)w_1\left((w_2)_-\right)^{p-2}(z_2-w_2).
	\end{eqnarray*}
	Since $z_1((z_2)_+)^{p-1}-z_1((z_2)_-)^{p-1}=z_1z_2^{\langle p-1\rangle}$,  it follows that
	\begin{align}
		\label{eq:Jppm}
		\mathcal{J}_p = \mathcal{J}_p^{(+)} - \mathcal{J}_p^{(-)} = \left(\mathcal{J}_p^{(+)}+\mathcal F_p\right) - \left(\mathcal{J}_p^{(-)}+\mathcal F_p\right),
	\end{align}
	where we consider $\mathcal F_p$ given by \eqref{eq:Fp} with $n=2$ and we have  $\mathcal{J}_p^{(+)}+\mathcal{F}_p\ge0$ and $\mathcal{J}_p^{(-)}+\mathcal{F}_p\ge0$ on $\left([0,\infty)\times\R\right)^2$.
	Note also that, if we denote $\bar{z}:=(z_1,-z_2)$, then
	\begin{align}
		\label{eq:Jpsym}
		\mathcal{J}_p^{(+)}(\bar{w},\bar{z})
		=
		\mathcal{J}_p^{(-)}(w,z).
	\end{align}
	
	Here is a preliminary version of Theorem~\ref{thm:polarized}.
	
	\begin{prop}
		\label{prop:HSpm}
		For $p>2$, $f,g\in\Lp$, $f\geq0$, and $\Phi(x)$, $\mP_t\Phi(x)$ as in \eqref{eq:denote}, 
		\begin{align}
			\label{eq:HSp}
			\intRd &\left( f(g_+)^{p - 1} + |\Phi|^p \right) \,\dx
			\nonumber \\
			&=
			\int_0^\infty \intRd \intRd
			\left(
			\mathcal{J}_p^{(+)}
			+ \mathcal{F}_p
			\right)
			(\mP_t\Phi(x),\mP_t\Phi(y))
			\nu(x,y)
			\,\dx\dy\dt
		\end{align}
		and
		\begin{align}
			\label{eq:HSm}
			\intRd &\left( f(g_-)^{p - 1} + |\Phi|^p \right) \,\dx
			\nonumber \\
			&=
			\int_0^\infty \intRd \intRd
			\left(
			\mathcal{J}_p^{(-)}
			+ \mathcal{F}_p
			\right)
			(\mP_t\Phi(x),\mP_t\Phi(y))
			\nu(x,y)
			\,\dx\dy\dt.
		\end{align}
	\end{prop}

	\begin{proof}
		We only  prove \eqref{eq:HSp} since  \eqref{eq:HSm}
		follows by substituting $-g$ in place of $g$,
		see \eqref{eq:Fpsym} and \eqref{eq:Jpsym}.
		The proof of \eqref{eq:HSp} is much alike that of Theorem~\ref{thm:HS15}. We use the convexity of $Y^{(+)}$, resulting in the nonnegativity of its Taylor remainder, the function $\mathcal J_p^{(+)}+\mathcal F_p$.
		As before, we first consider  $f,g\in\DpL$. Fix some $0\leq t\le T<\infty$. Let $u(t):=P_tf$, $v(t):=P_tg$, and $U(t):= \mP_t\Phi = (u(t),v(t))\in L^p(\Rd;\R^2)$. Actually, $U(t)\in(\DpL)^2$.
		As seen in the proof of Theorem~\ref{thm:HS15}, the function $t\mapsto|U(t)|^p$ is continuously differentiable in $L^1(\Rd)$ and $\left( |U(t)|^p \right)'=pU(t)^{\langle p-1 \rangle}\cdot  L U(t)$.
		Since $(v(t)_+)^{p-1}=(|v(t)|^{p-1}+v(t)^{\langle p-1 \rangle})/2$, from Corollary~\ref{lem:pdif} with $\kappa=p-1>1$, we obtain that $(v(t)_+)^{p-1}$ is continuously differentiable in $L^\frac{p}{p-1}(\Rd)$ and
		\begin{eqnarray*}
			\left( (v(t)_+)^{p-1} \right)'
			&=&
			\left( \frac{|v(t)|^{p-1}+v(t)^{\langle p-1 \rangle}}{2} \right)'
			=
			\frac{p-1}{2} \left(v(t)^{\langle p-2 \rangle} Lv(t) + |v(t)|^{p-2} Lv(t)\right)
			\\
			&=&
			(p-1) (v(t)_+)^{p-2} Lv(t).
		\end{eqnarray*}
		By Lemma~\ref{lem:product}, $u(t)(v(t)_+)^{p-1}$ is continuously differentiable in $L^1(\Rd)$ and
		\begin{align}\label{eq:diff-l1}
			\left( u(t)(v(t)_+)^{p-1} \right)'
			=
			(v(t)_+)^{p-1}  Lu(t)  + (p-1) u(t)(v(t)_+)^{p-2} Lv(t).
		\end{align}
		In particular, $\left(u(t)(v(t)_+)^{p-1}\right)'$ is
		 well-defined and continuous in $L^1(\R^d)$. As in
		\eqref{eq:ddtintmod},
		\begin{align}
			W(t)
			&:=
			\frac{\rm d}{\dt} \int\limits_\Rd \left(u(t)(v(t)_+)^{p-1} + |U(t)|^p\right)\,\dx
			=
			\int\limits_\Rd  \frac{\rm d}{\dt}\left[u(t)(v(t)_+)^{p-1} + |U(t)|^p \right]\,\dx
			\nonumber \\ \label{eq:LpRangle}
			&=
			\langle Lu(t), (v(t)_+)^{p-1} \rangle  + \langle Lv(t), (p-1) u(t)(v(t)_+)^{p-2} \rangle  + \langle L U(t), pU(t)^{\langle p-1 \rangle}\rangle.
		\end{align}
		Since the limits defining $Lu$,  $Lv$ (respectively, $LU$) exist strongly in $L^p(\Rd)$ (respectively, in $L^p(\Rd;\R^2)$) and $(v(t)_+)^{p-1}$,  $u(t)(v(t)_+)^{p-2}$ (respectively, $U(t)^{\langle p-1\rangle}$)
		belong to $L^q(\Rd)$ (respectively, to $L^q(\Rd;\R^2)$), we get
		\begin{align*}
			W(t)=\lim_{h\to 0^+}  \intRd\intRd & \Big((u(t)(y)-u(t)(x))(v(t){(x)}_+)^{p-1}\\
			& + (p-1) (v(t)(y)-v(t)(x))u(t)(x)(v(t)(x)_+)^{p-1}\\
			& +p(U(t)(y)-U(t)(x))\cdot U(t)(x)^{\langle p-1\rangle}\Big)
			\frac{p_h(x,y)}{h}\,\dx\dy.
		\end{align*}
		As $(P_t)_{t\geq 0}$ is conservative, for every $h>0$, we have
		\begin{align*}
			\intRd |U(t)|^p\,\dx = \intRd P_h \left(|U(t)|^p\right)\,\dx
		\end{align*}
		and
		\begin{align*}
			\intRd u(t)(v(t)_+)^{p-1} \,\dx = \intRd P_h\left(u(t)(v(t)_+)^{p-1}\right) \,\dx.
		\end{align*}
		Taking this into account and rearranging, we get 
		\[W(t)=\lim_{h\to 0^+} \intRd\intRd \left(\mathcal J_p^{(+)} +\mathcal F_p\right)(U(t)(x),U(t)(y))\,\frac{p_h(x,y)}{h}\,\dx\dy.\]
		Because of the assumption $f\ge   0$, we have $U(t)\in [0,\infty)\times\R$ for all $x\in\Rd$ and $t\geq 0$,  so that
		$\left(\mathcal J_p^{(+)} +\mathcal F_p\right)(U(t)(x),U(t)(y))$ is nonnegative
		(see the discussion preceding the proposition). Therefore
		from \eqref{eq:mp}, \eqref{eq:zp}, and  Lemma~\ref{lem:Fatou}, we conclude that
		\begin{equation}\label{eq:Wt}
			W(t)=\intRd\intRd \left(\mathcal J_p^{(+)} +\mathcal F_p\right)(U(t)(x),U(t)(y))\nu(x,y) \,\dx\dy.
		\end{equation}
		Since $u(t)(v(t)_+)^{p-1}+|U(t)|^p$ is continuously differentiable  in $L^1(\Rd)$ for  $t\in {[0,\infty)}$,  $W(t)$ is a continuous (real) function on $(0,\infty)$. Thus,
		\begin{eqnarray*}
			&&\intRd \Big( u(0)(v(0)_+)^{p - 1} + |U(0)|^p \Big)\,\dx
			-
			\intRd \left(u(T)((v(T))_+)^{p-1} + |U(T)|^p\right)\,\dx	\\
			&=&
		-\int_0^T W(t)\,\dt
			=
			\int_0^T \intRd \intRd
			\left(
			\mathcal{J}_p^{(+)}
			+ \mathcal{F}_p
			\right)
			(U(t)(x),U(t)(y))
			\nu(x,y)
			\,\dx\dy\dt\\
			&=&
			\int_0^T \intRd \intRd
			\left(
			\mathcal{J}_p^{(+)}
			+ \mathcal{F}_p
			\right)
			(P_t\Phi(x),P_t\Phi(y))
			\nu(x,y)
			\,\dx\dy\dt.
		\end{eqnarray*}	
		We now let $T\to\infty$. As the integrand in the right-hand side is nonnegative, $u(0)=f$,  $v(0)=g$, and $U(0)=\Phi$,   to prove \eqref{eq:HSp} it is enough to show that
		$$
		\intRd \left(u(T)((v(T))_+)^{p-1} + |U(T)|^p\right) \,\dx=\intRd \left(P_Tf((P_Tg)_+)^{p-1} + |\mP_T\Phi|^p\right) \,\dx \to 0.
		$$
		While proving  Theorem~\ref{thm:HS15},  we have already shown that $\intRd |U(T)|^p \,\dx \to 0$.  Further,
		since $|P_Tf(x)|\le f^*(x)$ and $|P_Tg(x)|\le g^*(x)$ for every $x\in \Rd$ and $T>0$ and $f^*,g^*\in\Lp$ by \eqref{eq:stein}, we get  $\intRd P_Tf((P_Tg)_+)^{p-1} \,\dx \to 0$ by the Dominated Convergence Theorem.
		This yields \eqref{eq:HSp} for $f,g\in\DpL$.
		
		It remains to  get rid of the assumption $f,g\in\DpL$. We proceed as in the proof of Theorem~\ref{thm:HS15}. Take  $f,g\in\Lp$  arbitrary and let $s>0$. Since $(P_t)_{t \geq 0}$ is an analytic semigroup on $\Lp$, $P_s f,P_s g\in\DpL$ as well. Consequently, by \eqref{eq:HSp},
		\begin{align*}
			\intRd & \left(P_sf ((P_sg)_+)^{p - 1} + |\mP_s\Phi|^p \right) \,\dx
			\\
			&=
			\int_s^\infty \intRd \intRd
			\left(
			\mathcal{J}_p^{(+)}
			+ \mathcal{F}_p
			\right)
			(P_t\Phi(x),P_t\Phi(y))
			\nu(x,y)
			\,\dx\dy\dt.
		\end{align*}
		Let $s\to0^+$.
		As the integrand of the right-hand side is nonnegative, the integrals tend to
		the right-hand side  of \eqref{eq:HSp}.

		To get the convergence of the left-hand side we use the strong continuity of $(P_t)_{ t \geq 0}$ in $\Lp$. The convergence $|\mP_s\Phi|^p\to|\Phi|^p$ in $L^1(\Rd)$ was shown in proof of Theorem~\ref{thm:HS15}.
		Since $P_sf\to f$  and
		$(P_sg)_+\to g_+$ in $\Lp$,
		by Lemma~\ref{lem:new-cont},
		$((P_sg)_+)^{p - 1}\to (g_+)^{p - 1}$ in $L^{\frac{p}{p-1}}(\Rd)$.
		Moreover, by  Lemma~\ref{thm:continuity}, $P_sf ((P_sg)_+)^{p - 1}\to f(g_+)^{p - 1}$ in $L^1(\Rd)$.
		Thus,
		$\intRd \left(P_sf ((P_sg)_+)^{p - 1} + |\mP_s\Phi|^p \right) \,\dx \to \intRd \left( f(g_+)^{p - 1} + |\Phi|^p \right) \,\dx$.
		The proof of \eqref{eq:HSp} is complete.
	\end{proof}

	\begin{proof}[Proof of Theorem~\ref{thm:polarized}]
		Thanks to Remark~\ref{r.p2}, we only need to consider  $p>2$. Let first $f\geq0$. By Proposition~\ref{prop:HSpm} and \eqref{eq:Jppm},
		\begin{align*}
			\intRd & fg^{\langle p - 1 \rangle} \,\dx
			=
			\intRd \left( f(g_+)^{p - 1} + |\Phi|^p \right) \,\dx
			-
			\intRd \left( f(g_-)^{p - 1} + |\Phi|^p \right) \,\dx
			\\
			&=
			\int_0^\infty \intRd \intRd
			\left(
			\mathcal{J}_p^{(+)}
			+ \mathcal{F}_p
			\right) (P_t\Phi(x),P_t\Phi(y))
			\nu(x,y)
			\,\dx\dy\dt
			\\
			&-
			\int_0^\infty \intRd \intRd
			\left(
			\mathcal{J}_p^{(-)}
			+ \mathcal{F}_p
			\right) (P_t\Phi(x),P_t\Phi(y))
			\nu(x,y)
			\,\dx\dy\dt
			\\
			&=
			\int_0^\infty \intRd \intRd
			\mathcal{J}_p (P_t\Phi(x),P_t\Phi(y))
			\nu(x,y)
			\,\dx\dy \dt,
		\end{align*}
		where all the integrals are absolutely convergent.
		Therefore \eqref{eq:int-abs} holds
		in this case.
		
		To get rid of  the assumption $f\geq0$, we consider an arbitrary $f\in\Lp$ and write $f=f_+-f_-$.
		The result holds for pairs $\Phi^{(+)}:=(f_+,g)$  and $\Phi^{(-)}:=(f_-,g)$.  Of course,
		$\Phi = \Phi^{(+)}-\Phi^{(-)}$.
		The operators $P_t$ are linear and the function $\mathcal J_p(w,z)$ is linear in $w_1$ and $z_1$, so
		\begin{eqnarray*}
			\intRd fg^{\langle p - 1 \rangle} \,\dx
			&=&
			\intRd \left(f_+\right)g^{\langle p - 1 \rangle} \,\dx
			-
			\intRd \left(f_-\right)g^{\langle p - 1 \rangle} \,\dx
			\\
			&=&
			\int_0^\infty \intRd \intRd
			\mathcal{J}_p (P_t\Phi^{(+)}(x),P_t\Phi^{(+)}(y))
			\nu(x,y)
			\,\dx\dy\dt
			\\
			&\quad & -
			\int_0^\infty \intRd \intRd
			\mathcal{J}_p (P_t\Phi^{(-)}(x),P_t\Phi^{(-)}(y))
			\nu(x,y)
			\,\dx\dy\dt
			\\
			&=&
			\int_0^\infty \intRd \intRd
			\mathcal{J}_p (P_t\Phi(x),P_t\Phi(y))
			\nu(x,y)
			\,\dx\dy\dt.
		\end{eqnarray*}
		The absolute convergence of the integrals is clear from our previous arguments.
	\end{proof}

We next present a quantitative version of \eqref{eq:int-abs}.
	
	\begin{prop}\label{coro:h-l} Under the assumptions of Theorem~\ref{thm:polarized},
		\begin{align}\label{eq:h-l}
			\int\limits_0^\infty\int\limits_\Rd\int\limits_\Rd \left|\mathcal J_p(P_t\Phi(x),P_t\Phi(y))\right| \nu(x,y) \,\dx\dy\dt
			\leq
			(1+2^{p/2})\|f\|_{\Lp}\|g\|_{\Lp}^{p-1}.
		\end{align}
	\end{prop}
	
	\begin{proof}
As in the proof of Theorem~\ref{thm:polarized}, we let $\Phi^{(+)}=(f_+,g)$ and $\Phi^{(-)}=(f_-,g)$.
		Then,
		\[
		\mathcal J_p(P_t\Phi(x),P_t\Phi(y)) =
		\mathcal J_p(P_t\Phi^{(+)}(x),P_t\Phi^{(+)}(y))-
		\mathcal J_p(P_t\Phi^{(-)}(x),P_t\Phi^{(-)}(y)),
		\]
		so
		\[
		\left|\mathcal J_p(P_t\Phi(x),P_t\Phi(y)) \right|\leq
		\left|\mathcal J_p(P_t\Phi^{(+)}(x),P_t\Phi^{(+)}(y))\right|+
		\left|\mathcal J_p(P_t\Phi^{(-)}(x),P_t\Phi^{(-)}(y))\right|.\]	
Because of \eqref{eq:Jppm},
		\[
		\left|\mathcal{J}_p\right| \leq \left(\mathcal{J}_p^{(+)}+\mathcal F_p\right) + \left(\mathcal{J}_p^{(-)}+\mathcal F_p\right),
		\]
		both terms being nonnegative on $\left([0,\infty)\times\R\right)^2$.
		As $P_t\Phi^{(+)}, P_t\Phi^{(-)}\in \left([0,\infty)\times\R\right)^2$,
		\begin{eqnarray*}
			\left|\mathcal{J}_p(P_t\Phi^{(+)}(x),P_t\Phi^{(+)}(y))\right|& \leq& \left(\mathcal{J}_p^{(+)}+\mathcal F_p\right)(P_t\Phi^{(+)}(x),P_t\Phi^{(+)}(y))\\
			&&+ \left(\mathcal{J}_p^{(-)}+\mathcal F_p\right)(P_t\Phi^{(+)}(x),P_t\Phi^{(+)}(y)),
		\end{eqnarray*}
		and a similar inequality holds for $P_t\Phi^{(-)}$.
		From Proposition~\ref{prop:HSpm},
		\[\int\limits_0^\infty\int\limits_\Rd\int\limits_\Rd (\mathcal J_p^{(+)}+\mathcal F_p)(P_t\Phi^{(+)}(x),P_t\Phi^{(+)}(y)) \nu(x,y) \,\dx\dy\dt= \int\limits_\Rd \left(f_+(g_+)^{p-1}+|\Phi^{(+)
		}|^p
		\right) \,\dx.\]
		A similar identity holds for $\mathcal J_p^{(-)}$.
Summing up, we get
		\[
		\int_0^\infty\intRd\intRd \left|\mathcal J_p(P_t\Phi^{(+)}(x),P_t\Phi^{(+)}(y))\right| \nu(x,y) \,\dx\dy\dt\leq \intRd \left(f_+|g|^{p-1}+2|\Phi^{(+)
		}|^p
		\right) \,\dx\]
		and 
		\[
		\int_0^\infty\intRd\intRd \left|\mathcal J_p(P_t\Phi(x),P_t\Phi(y))\right| \nu(x,y) \,\dx\dy\dt\leq \intRd \left(|f| |g|^{p-1}+2|\Phi
		|^p
		\right) \,\dx.\]		
By	H\"{o}lder inequality,
		\[\intRd f|g|^{p-1} \,\dx \leq \|f\|_{\Lp} \|g\|_{\Lp}^{(p-1)}.\]
		On the other hand,
		\[ |\Phi|^p= (f^2+g^2)^{p/2}\leq 2^{p/2-1}(|f|^p+|g|^p).\]
		Therefore, if $\|f\|_{\Lp}=\|g\|_{\Lp}=1$,  then \eqref{eq:h-l} is true if we replace its right-hand side by $1+2^{p/2}$.
		If $\|f\|_{\Lp}=0$ or $\|g\|_{\Lp}=0$,  then \eqref{eq:h-l} is obvious. Otherwise, we observe that $\mathcal J_p$ is homogeneous in the first coordinates, and $(p-1)$-homogeneous in the second, to wit,
		\[\mathcal J_p((\lambda w_1,\mu w_2),(\lambda z_1, \mu z_2))=
		\lambda \mu^{\langle p-1\rangle} \mathcal J_p((w_1,w_2), (z_1,z_2)), \qquad \lambda,\mu>0.\]
		Then, by considering $f/\|f\|_{\Lp}$ and $g/\|g\|_{\Lp}$, we get the result.
		
	\end{proof}

\section{Polarized Sobolev--Bregman form $\mathcal E_p(u,v)$}
\label{sec:form-polarized}

	The integral expression appearing in \eqref{eq:polarized} and Theorem~\ref{thm:polarized}, namely
	\begin{equation*}
		\E_p(u,v) := \frac{1}{p} \intRd\intRd \mathcal J_p(\Phi(x),\Phi(y))\nu(x,y)\,\dx\dy,  
	\end{equation*}
	where $\Phi(x)=(u(x),v(x))$, $u,v\!:\mathbb R^d\to\mathbb R$, $p\in [2,\infty)$, and $\mathcal J_p$
	is given by \eqref{eq:Jp},  deserves further attention. If $u=v$ then $\mathcal E_p(u,v)=\mathcal E_p(u,u)=\mathcal E_p[u]$. For $p=2$,  we get $\mathcal E_2(u,v)$, the usual (bilinear) Dirichlet form \cite{MR2778606}, in particular,  it is symmetric. For $p>2$,  in general $\E_p(v,u)\neq \E_p(u,v)$ and we are even puzzled by
	the question whether the integral in \eqref{eq:polarized} is well-defined for general enough functions $u,v$, for instance 
for $u,v\in \mathcal D(\mathcal E_p)$.	
	The next theorem  asserts that for $p\geq2$ and $u,v\in\DpL$, \eqref{eq:polarized} is well-defined; we also get an extension of the single-function formula \eqref{eq:pfag} from Proposition~\ref{thm:EpDpL}.
	
	\begin{thm}
		\label{thm:EpDpL2}
		Let $p\geq2$. If $u,v\in\DpL$, then $\mathcal E_p(u,v)$ is well-defined and
		\begin{align}
			\label{eq:pfag2}
			\mathcal E_p(u,v)
			=
			-\frac{1}{p} \langle Lu, v^{\langle p-1 \rangle} \rangle
			-\frac{1}{p} \langle Lv, (p-1)u|v|^{p-2} \rangle.
		\end{align}
	\end{thm}
	
	Note that this agrees with \eqref{eq:pfag} if $u=v$.
	Before we prove \eqref{eq:pfag2}, we need to further decompose $\mathcal J_p^{(+)}$ (and $\mathcal J_p$)  into a difference of  nonnegative functions.
	
	Let $\indyk(a):=(1+\sgn(a))/2$ be the Heaviside step function. We define
	\begin{eqnarray}\label{eq:jotpp}
		\mathcal{J}_p^{(++)}(w,z)
		&:=&
		(z_1)_+\left((z_2)_+\right)^{p-1} - (w_1)_+\left((w_2)_+\right)^{p-1}-\indyk(w_1)\left((w_2)_+\right)^{p-1}(z_1-w_1)
		\nonumber
		\\
		&&-(p-1)(w_1)_+\left((w_2)_+\right)^{p-2}(z_2-w_2),
		\\
		\label{eq:jotmp}
		\mathcal{J}_p^{(-+)}(w,z)
		&:=&
		(z_1)_-\left((z_2)_+\right)^{p-1} - (w_1)_-\left((w_2)_+\right)^{p-1}+\indyk(-w_1)\left((w_2)_+\right)^{p-1}(z_1-w_1)
		\nonumber
		\\
		&&-(p-1)(w_1)_-\left((w_2)_+\right)^{p-2}(z_2-w_2),
	\end{eqnarray}
	where $w:=(w_1,w_2), z:=(z_1,z_2)\in \R^2$.
	We may view these functions as the second-order Taylor remainders of the mappings $\R^2\ni z\mapsto (z_1)_+\left((z_2)_+\right)^{p-1}$ and  $\R^2\ni z\mapsto (z_1)_-\left((z_2)_+\right)^{p-1}$, respectively, except for nondifferentiability of the mappings on the vertical positive semi-axis (for more details, see the proof of Lemma~\ref{lem:Jpppmmpositive} in Appendix~\ref{sec:apx}).
	
	Similarly to \eqref{eq:Jppm} and \eqref{eq:Jpsym}, we get a decomposition of $\mathcal J_p^{(+)}$:
	\begin{align}
		\label{eq:Jpppmm}
		\mathcal{J}_p^{(+)}= \mathcal{J}_p^{(++)} - \mathcal{J}_p^{(-+)}
	\end{align}
	and the identity
	\begin{align}
		\label{eq:Jpsym2}
		\mathcal{J}_p^{(++)}(-\bar{w},-\bar{z})
		=
		\mathcal{J}_p^{(-+)}(w,z).
	\end{align}
	In Lemma~\ref{lem:Jpppmmpositive} in Appendix~\ref{sec:apx} we prove that
	\begin{align*}
		\mathcal{J}_p^{(++)}(w,z)+\mathcal{F}_p(w,z) \ge 0, \quad
		\mathcal{J}_p^{(-+)}(w,z)+\mathcal{F}_p(w,z) \ge 0
	\end{align*}
	for all $z,w\in\R^2$.
	Therefore, by adding and subtracting $\mathcal F_p$ in  \eqref{eq:Jpppmm}, we get the desired decomposition of $\mathcal J_p^{(+)}$ and we can proceed from there.
	Let us mention that it is crucial to define the Heaviside function so that $\indyk(0)=1/2$. This is because we use the identity $\indyk(a)+\indyk(-a)=1$ for all $a\in\R$ to derive \eqref{eq:Jpppmm}.

	\begin{proof}[Proof of Theorem~\ref{thm:EpDpL2}.]
		Let $u,v\in\mathcal D_p(L)$.  If $p=2$,
		then, again,  the identity is evident from
		classical polarization,
		\eqref{eq:E2} and \eqref{eq:pfag}.
		
		Thus, we let $p>2$. Denote $\Phi(x):=(u(x),v(x))$. First we prove
		the following:
		\begin{eqnarray}
			\label{eq:lpp}
			l_{++}
			&:=&
			-\langle Lu,\indyk(u)(v_+)^{p-1}\rangle
			-
			\langle Lv,(p-1)u_+(v_+)^{p-2}\rangle
			-
			\langle  L\Phi,p\Phi^{\langle p-1 \rangle}\rangle
			\\ \nonumber
			&\ =&
			\intRd\intRd
			(\mathcal J_p^{(++)}+\mathcal F_p)(\Phi(x),\Phi(y))
			\nu(x,y)\,\dx\dy
		\end{eqnarray}
		and
		\begin{eqnarray}
			\label{eq:lmp}
			l_{-+}
			&:=&
			\langle Lu,\indyk(-u)(v_+)^{p-1}\rangle
			-
			\langle Lv,(p-1)u_-(v_+)^{p-2}\rangle
			-
			\langle  L\Phi,p\Phi^{\langle p-1 \rangle}\rangle
			\\ \nonumber
			&\ =&
			\intRd\intRd
			(\mathcal J_p^{(-+)}+\mathcal F_p)(\Phi(x),\Phi(y))
			\nu(x,y)\,\dx\dy.
		\end{eqnarray}
		We start with the proof of \eqref{eq:lpp}.
		By the definition of $ L$,
		\begin{align}
			\langle  L\Phi,p\Phi^{\langle p-1 \rangle}\rangle
			= \lim_{h\to 0^+}
			\intRd \intRd
			(\Phi(x) - \Phi(y)) \cdot p\Phi(x)^{\langle p-1 \rangle}
			\frac{p_h(x,y)}{h}
			\,\dx\dy.
		\end{align}
		Since the limits defining $Lu$,  $Lv$ exist in the strong sense in $L^p(\Rd)$,  we have
		\begin{align*}
			\nonumber
			l_{++}
			=&
			\lim_{h\to0^+}
			\intRd\intRd\left[(u(x)-u(y))\indyk(u(x))(v(x)_+)^{p-1}
			\right.
			\\ \nonumber
			&\quad\qquad + (v(x)-v(y))(p-1)u(x)_+(v(x)_+)^{p-2}
			\\
			&\quad\qquad \left. + (\Phi(x)-\Phi(y))\cdot p \Phi(x)^{\langle p-1\rangle}\right] \frac{p_h(x,y)}{h}
			\,\dx\dy.
		\end{align*}
		Then, similarly as in the proofs of Theorems~\ref{thm:HS15} and \ref{thm:polarized}, we take advantage of the conservativeness of the semigroup $(P_t)_{t\geq 0}$:
		\begin{eqnarray*}
			\intRd u_+(v_+)^{p-1}\,\dx &=& \intRd P_h\left( u_+(v_+)^{ p-1}\right)\,\dx,
			\\
			\intRd|\Phi|^p\,\dx &=& \intRd P_h\left(|\Phi|^p\right)\,\dx,\quad\mbox{ for } h>0.
		\end{eqnarray*}
		Taking this into account and rearranging, we obtain
		\[
		l_{++} = \lim_{h\to 0^+} \intRd\intRd (\mathcal J_p^{(++)}+\mathcal F_p)(\Phi(x),\Phi(y))\,\frac{p_h(x,y)}{h}\,\dx\dy.
		\]
		From Lemma~\ref{lem:Jpppmmpositive} in Appendix~\ref{sec:apx}, $\mathcal{J}_p^{(++)}+\mathcal{F}_p\ge 0$, hence we can pass to the limit as $h\to 0^+$ and by Lemma~\ref{lem:Fatou} we obtain \eqref{eq:lpp}.
		By substituting $-u$ in place of $u$, we obtain \eqref{eq:lmp}, too; see \eqref{eq:Fpsym} and \eqref{eq:Jpsym2}.

		Further, we claim that
		for all $u,v\in\DpL$,
		\begin{eqnarray}
			\label{eq:lp}
			l_+
			&:=&
			-\langle Lu,(v_+)^{p-1}\rangle
			-
			\langle Lv,(p-1)u(v_+)^{p-2}\rangle
			\\ \nonumber
			&\ =&
			\intRd\intRd
			\mathcal J_p^{(+)}(\Phi(x),\Phi(y))
			\nu(x,y)\,\dx\dy
		\end{eqnarray}
		and
		\begin{eqnarray}
			\label{eq:lm}
			l_-
			&:=&
			-\langle Lu,(v_-)^{p-1}\rangle
			+
			\langle Lv,(p-1)u(v_-)^{p-2}\rangle
			\\ \nonumber
			&\ =&
			\intRd\intRd
			\mathcal J_p^{(-)}(\Phi(x),\Phi(y))
			\nu(x,y)\,\dx\dy.
		\end{eqnarray}
		Indeed, using \eqref{eq:lpp}, \eqref{eq:lmp}, and \eqref{eq:Jpppmm}, we get
		\begin{eqnarray*}
			l_+
			=
			l_{++} - l_{-+}
			&=&
			\intRd\intRd
			(\mathcal J_p^{(++)}+\mathcal F_p)(\Phi(x),\Phi(y))
			\nu(x,y)\,\dx\dy
			\\
			&&
			-
			\intRd\intRd
			(\mathcal J_p^{(-+)}+\mathcal F_p)(\Phi(x),\Phi(y))
			\nu(x,y)\,\dx\dy
			\\
			&=&
			\intRd\intRd
			\mathcal J_p^{(+)}(\Phi(x),\Phi(y))
			\nu(x,y)\,\dx\dy.
		\end{eqnarray*}
		Note that the integral on the right-hand side is well-defined  as a difference of finite integrals with nonnegative integrands. This yields \eqref{eq:lp}.
		Equality \eqref{eq:lm} follows from \eqref{eq:lp} by substituting $-v$ in place of $v$; see \eqref{eq:Fpsym} and \eqref{eq:Jpsym}.
		
		To conclude, using \eqref{eq:lp}, \eqref{eq:lm}, and \eqref{eq:Jppm}, we obtain
		\begin{align}\nonumber
			-\langle &Lu, v^{\langle p-1 \rangle} \rangle
			-\langle Lv, (p-1)u|v|^{p-2} \rangle
			=
			l_+ - l_-
			=
			\\
			&=
			\intRd\intRd
			\mathcal J_p(\Phi(x),\Phi(y))
			\nu(x,y)\,\dx\dy =
			p\mathcal E_p(u,v).\label{e.dvJp}
		\end{align}
		Again,  the integral defining $\mathcal E_p(u,v)$ is absolutely convergent as a difference of two absolutely convergent integrals. The proof is complete.
	\end{proof}

\begin{rem}\label{r.sum}
By the above and Lemma~\ref{lem:otco}, 
\begin{eqnarray*}	
			p\mathcal E_p(f,g)
			&:=&
			\intRd\intRd
			\mathcal J_p(f(x),g(x); f(y), g(y))
			\nu(x,y)\,\dx\dy
			\\
			&\ =&
			\langle -Lf, g^{\langle p-1 \rangle} \rangle +
			\langle -Lg, (p-1)f|g|^{p-2} \rangle
			\\
			&\ =&
			-\frac{\rm d}{\dt}\intRd P_tf(x) (P_tg(x))^{\langle p - 1 \rangle} \,\dx\Big|_{t=0} ,
		\end{eqnarray*}
at least for $f,g\in \DpL$ and $p\geq 2$.
At this moment, Lemma~\ref{lem:otco} offers a simplifying perspective on \eqref{eq.HScd} and Theorem~\ref{thm:polarized}, but we should emphasize the importance of absolute integrability asserted in Theorem~\ref{thm:polarized} for arbitrary $f,g\in L^p(\Rd)$ when $p\ge2$; see also Proposition~\ref{coro:h-l}.
\end{rem}

	\smallskip

	\appendix
	
	\section{Estimates for Bregman divergence}\label{sec:App-A}

	The following lemma extends Lemma 5 of \cite{MR4372148}, where scalar versions of \eqref{eq:restpowers}, \eqref{eq:difpowers}, \eqref{eq:diffrench} were given. The inequality \eqref{eq:restfrench} seems new.
	\begin{lem}\label{lem:inequalities-1}
		There are constants $C_\kappa,C_\kappa',C_\kappa'',C_\kappa'''\in(0,\infty)$ such that for all $w,z\in\Rn$,
		\begin{align}
			\label{eq:restpowers}
			0\leq \mathcal F_\kappa(w,z)
			&\leq
			C_\kappa |z-w|^\lambda ({|w|\vee |z|})^{\kappa-\lambda}, \quad \lambda\in[0,2], \kappa>1,
			\\
			\label{eq:restfrench}
			|\mathcal F_{\langle\kappa\rangle}(w,z)|
			&\leq
			C_\kappa' |z-w|^\lambda ({|w|\vee |z|})^{\kappa-\lambda},\quad \lambda\in[0,2],  \kappa>1,
			\\
			\label{eq:difpowers}
			||z|^\kappa - |w|^\kappa|
			&\leq
			C_\kappa'' |z-w|^\lambda ({|w|\vee |z|})^{\kappa-\lambda}, \quad \lambda\in[0,1], \kappa>0,
			\\
			\label{eq:diffrench}
			|z^{\langle \kappa \rangle} - w^{\langle \kappa \rangle}|
			&\leq
			C_\kappa'''
			|z-w|^\lambda ({|w|\vee |z|})^{\kappa-\lambda},  \quad \lambda\in[0,1], \kappa>0.
		\end{align}
	\end{lem}
	\begin{proof}
		It suffices to prove the inequalities for the maximal
		value of $\lambda$ (equal to 2 in \eqref{eq:restpowers},  \eqref{eq:restfrench}, and equal to 1 in \eqref{eq:difpowers}, \eqref{eq:diffrench}). Indeed, for other values of $\lambda$,  it is enough to use the inequality $|z-w|\le|z-w|^\mu(|z|+|w|)^{1-\mu}$,  $\mu\in (0,1), $
		$w,z\in\R^n$.
		
		Inequality \eqref{eq:restpowers} follows from \eqref{eq:Fp2Gp2}. In particular, for $a,b\in\R$,  $\lambda=2$, we have
		\begin{equation}\label{eq:real-ineq}
			0\leq F_\kappa(a,b)=|b|^\kappa-|a|^\kappa-\kappa a^{\langle \kappa-1\rangle}(b-a)\leq C_\kappa|b-a|^2(|b|\vee |a|)^{\kappa-2}.
		\end{equation}
		To get the other inequalities, observe that they are obvious for $w=0$. For $w\neq 0$, we divide by $|w|^\kappa$ and, denoting $t:=|z|/|w|\in [0,\infty)$,  $v_1:= w/|w|\in \Sn$,  $v_2:=z/|z|\in \Sn$,  we arrive at the following equivalent statements
		of \eqref{eq:restfrench}, \eqref{eq:difpowers}, \eqref{eq:diffrench}:
		\begin{align}   \label{eq:nowe-2}
			|t^\kappa v_2- v_1-\left((\kappa-1)v_1\otimes v_1 +\mbox{Id}\right)(tv_2-v_1)|&\leq
			C_\kappa'|tv_2-v_1|^2 (1\vee t)^{\kappa -2},\\
			\label{eq:nowe-3}
			{|t^\kappa-1|}&\leq C_\kappa''{|tv_2-v_1|} (1\vee t)^{\kappa -1},\\
			\label{eq:nowe-4}
			{|t^\kappa v_2- v_1|}&\leq C_\kappa'''{|tv_2-v_1|}(1\vee t)^{\kappa -1}.
		\end{align}
		We have
		\begin{equation}\label{eq:nowe-right1}
			{|tv_2-v_1|}^2(1\vee t)^{\kappa-2}=\left((1-t)^2+2t(1-( v_1,v_2))\right)(1\vee t)^{\kappa-2}.
		\end{equation}
		If we square the right-hand sides of \eqref{eq:nowe-3} and \eqref{eq:nowe-4} then, up to a constant, we get
		\begin{equation}\label{eq:nowe-right2}
			{|tv_2-v_1|}^2(1\vee t)^{2\kappa-2}=\left((1-t)^2+2t(1-( v_1,v_2))\right)(1\vee t)^{2\kappa-2}.
		\end{equation}
		Denote $\beta=1-( v_1,v_2)\in [0,2]$,  so that  \eqref{eq:nowe-2} becomes
		\begin{equation}
			\label{neq:nowe-2.2}
			|(t^\kappa-t)v_2 + (\kappa-1)\left((1-t)+\beta t\right)v_1|\leq C_\kappa'\left((1-t)^2+2\beta t\right)({1\vee t})^{\kappa -2}.
		\end{equation}
		This inequality is evident when $t$ is away from 1, say, $t\in [0,\frac{1}{2}]$ or $t\in [2,\infty)$.
		Indeed, for $t\leq 1/2$,  we estimate the left-hand side by ${2}\kappa$,  while the function on the right-hand side is not smaller than $\left(\frac{1}{2}\right)^2$, and \eqref{eq:nowe-2} follows.
		When $t\geq 2$,  then the left-hand side is not {greater} than ${(2\kappa-1)}t^\kappa$,  and for the right-hand side, we get
		\[\left((1-t)^2+ 2\beta t\right)({1\vee t})^{\kappa-2} \geq
		\left(\frac{t}{2}\right)^2 t^{\kappa-2}.\]
		
		To deal with the remaining range $t\in (1/2,2)$,  we square both sides of \eqref{neq:nowe-2.2}.
		The left-hand side yields
		\begin{eqnarray}
			\nonumber
			&&|(t^\kappa -t)v_2 + (\kappa -1)((1-t) +t(1-{(} v_1,v_2{)}))v_1 |^2\\
			\label{eq:nowe-2-1}
			&=& |F_\kappa(1,t)+ (\kappa -1)\beta t|^2 -2(t^\kappa-t)(\kappa-1)
			((1-t)+t\beta))\beta.
		\end{eqnarray}
		In view of \eqref{eq:real-ineq}, the first term on the right-hand side of \eqref{eq:nowe-2-1} is bounded above by $(C_\kappa(1-t)^2 + (\kappa-1)t\beta)^2$.
		Since
		the right-hand side of \eqref{eq:nowe-2}
		is then  not smaller than  a constant multiple of $((1-t)^2 +2\beta t)$, we get the estimate of this part.
		
		For the other term, we use the estimate $\left|1- t^{\kappa-1}\right|\leq C(\kappa)\left|1-t\right|$, $t\in(1/2,2]$, so
		$$2(t-t^\kappa)(\kappa-1)((1-t)+t\beta )\beta \leq C [(t-1)^2 t\beta + 4t^2\beta^2]\leq
		C((1-t)^2+2\beta t)^2.$$
		The estimate \eqref{eq:nowe-2} follows.
		
		After squaring its sides, the proof of \eqref{eq:nowe-4} is reduced to verifying
		\[({t^\kappa-1})^2+2\beta t^\kappa \leq C \left((1-t)^2+ 2\beta t\right)(1\vee t)^{2\kappa-2},\]
		with a constant $C$, uniformly in $\beta\in[0,2]$. This is done like before. For $t\geq 1$,
		\[({t^\kappa-1})^2\leq C (1-t)^2 t^{2\kappa-2}.\]
		For $0\leq t\leq 1/2$, the left-hand side is bounded and the right-hand side is bounded away from zero (uniformly in $\beta\in[0,2])$,  while
		for $t\in (1/2,1]$ we use the inequality $t^\kappa\leq C(\kappa)t$,  $t\in (1/2,1)$.  The square of the left-hand side of \eqref{eq:nowe-3} is smaller than $(1-t)^2 +2\beta t^\kappa$, i.e., the square of the left-hand side of \eqref{eq:nowe-4}. The proof is complete.
	\end{proof}

	\section{Calculus in $L^p$}\label{sec:App-B}
	Let $p \in [1,\infty)$ be fixed. In the discussion of the multivariate Hardy\nobreakdash--Stein identity above we use the  differential calculus in the Banach space
	\[L^p(\R^d;\R^n) := \left\{\Upsilon\!:\R^d\to\R^n \mbox{ measurable, } \intRd|\Upsilon(x)|^p\,\dx<\infty\right\},\quad n=1,2,\ldots,\]
	with the norm
	$\|\Upsilon\|_{L^p(\Rd;\Rn)}:=\left(\int_{\R^d}|\Upsilon(x)|^p\,\dx\right)^{1/p}$,
	or
	\[\|\Upsilon\|_{\ell_2^n(L^p(\R^d))} := \left(\sum_{i=1}^n \|\upsilon_i\|_{L^p(\Rd)}^2\right)^{1/2},\]
	where $\Upsilon=(\upsilon_1,\ldots,\upsilon_n)$, $\upsilon_1,\ldots,\upsilon_n\in L^p(\Rd)$.
	The norms are comparable:
	\begin{eqnarray*}
		\left(\intRd|\Upsilon(x)|^p\,\dx \right)^\frac{1}{p}
		&=&
		\left( \intRd \left(\sum_{i=1}^n |\upsilon_i(x)|^2\right)^\frac{p}{2}\,\dx\right)^\frac{1}{p}
		\leq
		\left( \intRd \left(\sum_{i=1}^n |\upsilon_i(x)|\right)^{p}\,\dx\right)^\frac{1}{p}\\
		&=&
		\||\upsilon_1|+\ldots+|\upsilon_n|\|_{\Lp}\leq \sum_{i=1}^n \|\upsilon_i\|_{\Lp} \leq \sqrt n \|\Upsilon\|_{\ell_2^n(L^p(\R^d))}.
	\end{eqnarray*}

	\smallskip

	Let $\Upsilon\in L^p(\Rd;\R^n)$ and $\Psi\in L^q(\Rd;\R^n)$, where $p,q\in(1,\infty)$ with $\frac{1}{p}+\frac{1}{q}=1$. We consider the canonical pairing
	\begin{align}
		\langle \Upsilon,\Psi \rangle := \intRd \Upsilon(x) \cdot \Psi(x) \,\dx
		= \sum_{j=1}^n \intRd \upsilon_j(x) \psi_j(x) \,\dx.
	\end{align}
	
	\smallskip
	
	For a mapping
	$[0,\infty) \ni t \mapsto \Upsilon(t) \in L^p(\Rd;\Rn)$,
	we denote
	$$\Delta_h\Upsilon(t) = \Upsilon(t+h)-\Upsilon(t) \quad \textnormal{provided} \,\, t, t+h\ge 0.$$
	As usual,
	$\Upsilon$ is called \emph{continuous}
	at $t_0\ge 0$ if $\Delta_h \Upsilon(t_0) \to 0$  in $L^p(\R^d;\R^n)$ as $h\to 0$.
	Furthermore,
	$\Upsilon$ is called \emph{differentiable}
	at $t_0\ge 0$ if the limit
	\begin{equation}\label{eq:def-diff}
		\lim_{h \to 0} \frac{1}{h} \Delta_h \Upsilon(t_0) =:\Upsilon'(t_0)
	\end{equation} exists in $L^p(\R^d;\R^n)$.
	If $\Upsilon'(t)$ defined by \eqref{eq:def-diff} is continuous at $t=t_0$,  then we say that $\Upsilon$ is {\em continuously differentiable} at $t_0$.
	In other words, $\Upsilon'(t_0)$ is the Fr\'{e}chet derivative of the mapping $[0,\infty)\ni t \mapsto \Upsilon(t)$ at $t_0$;
$\Upsilon'(0)$ denotes the right-hand side derivative at $0$.
		Clearly, if $\Upsilon$ is continuously differentiable on $[0,\infty)$, then $\Upsilon$ is continuous on $[0,\infty)$.
	
	Of course, $\Upsilon=(\upsilon_1,\ldots,\upsilon_n)$ is continuous
	(respectively, differentiable, continuously differentiable) in $L^p(\Rd;\Rn)$
	if and only if all the functions
	$\upsilon_i, $ $i=1,\ldots,n$,
	are	continuous
	(respectively, differentiable, continuously differentiable) in $L^p(\Rd)$.
	
	\smallskip
	
	We next present a series of auxiliary lemmas.
	
	\begin{lem}\label{lem:new-cont}
		Let $\kappa\in(0,p]$.  Then the following mappings are continuous:
		\begin{eqnarray}
			\label{eq:new-cont-1}
			L^p(\Rd;\Rn)\ni\Upsilon&\mapsto & \Upsilon^{\langle\kappa\rangle}\in L^{p/\kappa}(\Rd;\Rn), \\
			\label{eq:new-cont-2}
			L^p(\Rd;\Rn)\ni\Upsilon&\mapsto & |\Upsilon|^\kappa\in L^{p/\kappa}(\Rd).		\end{eqnarray}
	\end{lem}

	\begin{proof}
		First, observe that
		$|\Upsilon|^\kappa$ and $\Upsilon^{\langle\kappa\rangle}$ are in $L^{\frac{p}\kappa}(\Rd)$ and $L^{p/\kappa}(\Rd;\Rn)$, respectively, 
		if
		$\Upsilon \in L^p(\Rd;\Rn)$.
		
		To prove \eqref{eq:new-cont-1}, choose $\lambda\in (0,1)$ such that $\kappa-\lambda>0$ and
		suppose $\Upsilon_k\to \Upsilon$ in $L^p(\Rd;\Rn)$ as $k\to\infty$.
		From \eqref{eq:diffrench} we get, for every $x\in\Rd$,
		\begin{eqnarray*}
			|\Upsilon_k(x)^{\langle \kappa\rangle} -\Upsilon(x)^{\langle \kappa\rangle}| &\leq& C_\kappa'''|\Upsilon_k(x) -\Upsilon(x)|^\lambda (|\Upsilon_k(x)|\vee|\Upsilon(x)|)^{\kappa-\lambda}.
		\end{eqnarray*}
		Using H\"{o}lder's inequality with exponents $\kappa/\lambda$ and $\kappa/(\kappa-\lambda)$, we get
		\begin{eqnarray*}
			\|\Upsilon_k^{\langle \kappa\rangle} - \Upsilon^{\langle \kappa\rangle}\|_{L^{p/\kappa}(\Rd;\R^n)}^{\kappa/p}
			&= &
			\intRd |\Upsilon_k(x)^{\langle \kappa\rangle} -\Upsilon(x)^{\langle \kappa\rangle}|^{\frac{p}{\kappa}}\,\dx\\
			&\leq &
			\intRd \left(C_\kappa'''\right)^{p/\kappa}|\Upsilon_k(x) -\Upsilon(x)|^{\frac{\lambda p}{\kappa}} (|\Upsilon_k(x)|\vee|\Upsilon(x)|)^{\frac{(\kappa-\lambda)p}{\kappa}}\,\dx\\
			&\leq & C
			\|\Upsilon_k-\Upsilon\|_{L^p(\R^d;\R^n)}^{p\lambda/\kappa}
			\cdot
			\left( \|\Upsilon_k\|_{\Lp}^{p(\kappa-\lambda)/\kappa}+\|\Upsilon\|_{\Lp}^{p(\kappa-\lambda)/\kappa}\right).
		\end{eqnarray*}
		The result follows.
		The proof of \eqref{eq:new-cont-2} is similar.
	\end{proof}
	The following generalization of \cite[Lemma 13]{MR4372148} follows from H\"{o}lder's inequality.
	\begin{lem}\label{thm:continuity}
		Let $q \in (1, \infty)$, $r \in \big[\frac{q}{q-1}, \infty \big)$,
		$\Upsilon\in L^q(\Rd;\Rn)$, $\Psi\in L^r(\Rd;\Rn)$. Then
		\[\| \Upsilon\cdot\Psi\|_{L^{\frac{qr}{q+r}}(\Rd;\Rn)} \leq \|\Upsilon\|_{L^p(\Rd;\Rn)} \|\Psi\|_{L^r(\Rd;\Rn)}.\]
		Moreover, if $\Upsilon_n \to \Upsilon$ in $L^q(\Rd;\Rn)$ and $\Psi_n \to \Psi$ in $L^r(\Rd;\Rn)$, then $\Upsilon_n\cdot\Psi_n \to \Upsilon\cdot\Psi$ in $L^{\frac{qr}{q+r}}(\Rd;\R^n)$, as $n\to \infty$.
	\end{lem}

	The next lemma is an extension of \cite[Lemma 15]{MR4372148}, where the result for $|\Upsilon|^\kappa$ was proved for $n=1$,  $\kappa=p$.
	
	\begin{lem}
		\label{thm:difpow}
		Let $1<\kappa\leq p<\infty$ be given.  If $[0,\infty)\ni t \mapsto \Upsilon(t)$ is continuously differentiable in $L^p(\Rd;\Rn)$,  then:
		\begin{enumerate}
			\item[(i)] $|\Upsilon|^\kappa$ is continuously differentiable in $L^{p/\kappa}(\Rd)$ and
			\begin{equation}\label{eq:diff-1}
				(|\Upsilon|^\kappa)' = \kappa \Upsilon^{\langle \kappa - 1 \rangle} \cdot \Upsilon',
			\end{equation}
			\item[(ii)] $\Upsilon^{\langle\kappa\rangle}$ is continuously differentiable in $L^{p/\kappa}(\Rd;\Rn)$ and
			\begin{equation}\label{eq:diff-2}
				\left(\Upsilon^{\langle\kappa\rangle}\right)'= \left(J_{\langle\kappa\rangle}\circ\Upsilon\right)\Upsilon',
			\end{equation}
			with $J_{\langle\kappa\rangle}$ defined in \eqref{eq:diff-french-n}.
		\end{enumerate}
	\end{lem}
	\begin{proof}
		Both statements are proved similarly, therefore we only prove \eqref{eq:diff-2}, as it is the more complicated of the two.
		
	We choose $\gamma\in(0,\kappa-1)$. Observe that when
        $\Upsilon\neq0$,
        \begin{eqnarray*}
            \left(J_{\langle\kappa\rangle}\circ\Upsilon\right)\Upsilon'
            &=&
            |\Upsilon|^{\kappa-1} \left((\kappa-1)(
			\frac{\Upsilon}{|\Upsilon|}\otimes\frac{\Upsilon}{|\Upsilon|})+\mbox {Id}\right)\Upsilon'
            \\
            &=&
            (\kappa-1) |\Upsilon|^{\kappa-1} (\frac{\Upsilon}{|\Upsilon|}\cdot \Upsilon') \frac{\Upsilon}{|\Upsilon|}
            +
            |\Upsilon|^{\kappa-1} \Upsilon'
            \\
            &=&
            (\kappa-1) (\Upsilon^{\langle \gamma \rangle}\cdot \Upsilon') \Upsilon^{\langle \kappa - 1- \gamma \rangle}
            +
            |\Upsilon|^{\kappa-1} \Upsilon'
        \end{eqnarray*}
        and
        $\left(J_{\langle\kappa\rangle}\circ\Upsilon\right)\Upsilon' = 0$
        whenever $\Upsilon=0$.
        Given that $\Upsilon$ and $\Upsilon'$ are continuous, it follows that the above mapping is also continuous, according to Lemmas \ref{lem:new-cont} and \ref{thm:continuity}.

        Fix $t\ge0$.  
		To prove \eqref{eq:diff-2}, we only need  to verify that for $h\to 0$,
		\begin{align*}
			W_h(t) := \frac{1}{h}\Delta_h \Upsilon^{\langle\kappa\rangle}(t) - (J_{\langle\kappa\rangle}\circ\Upsilon(t)) \frac{1}{h}\Delta_h u(t) \to 0\qquad \mbox{in}\ L^{p/\kappa}(\Rd;\R^n).
		\end{align*}
		Since $W_h(t)=\frac{1}{h} \mathcal F_{\langle\kappa\rangle}(\Upsilon(t),\Upsilon(t+h))$,  we  choose $\lambda\in(1,2]$ such that $\kappa-\lambda>0$,  then use
		the inequality \eqref{eq:restfrench}  to get:
		\begin{eqnarray*}
			|W_h(t)|
			&\leq&
			\frac{1}{|h|} C_\kappa'
			|\Upsilon(t+h)-\Upsilon(t)|^\lambda
			(|\Upsilon(t+h)|\vee|\Upsilon(t)|)^{\kappa-\lambda}
			\\
			&=&
			|h|^{\lambda-1} C_\kappa'
			\left|\frac{1}{h}\Delta_h \Upsilon(t)\right|^\lambda
			(|\Upsilon(t+h)|\vee|\Upsilon(t)|)^{\kappa-\lambda}.
		\end{eqnarray*}
		Furthermore, by H\"{o}lder's inequality with parameters $\kappa/\lambda$ and $\kappa/(\kappa-\lambda)$,
		\begin{eqnarray*}
			\|W_h(t)\|_{{L^{p/\kappa}(\Rd;\R^n)}}
			& \leq &
			C_\kappa |h|^{\lambda-1}
			\left\|
			\left|\frac{1}{h}\Delta_h \Upsilon(t)\right|^\lambda
			(|\Upsilon(t+h)|\vee|\Upsilon(t)|)^{\kappa-\lambda}
			\right\|_{L^{p/\kappa}(\Rd)} \nonumber
			\\
			& \leq &
			C_\kappa |h|^{\lambda-1}\left\|\frac{1}{h}\Delta_h\Upsilon(t)
			\right\|_{L^p(\Rd;\R^n)}^\lambda \cdot \left\||\Upsilon(t+h)|+|\Upsilon(t)|\right\|_{\Lp}^{\kappa-\lambda}.
		\end{eqnarray*}
		We then conclude as in the proof of Lemma~\ref{lem:new-cont}.
	\end{proof}
	
	Finally  we invoke, without proof, an analogue of the Leibniz rule.
	\begin{lem}[\bf Product rule]
		\label{lem:product}
		Let $p>1$ and $r\in \left[\frac{p}{p-1},\infty\right)$ be given.
		If the mappings $[0,\infty)\ni t \mapsto \Upsilon(t)\in L^{p}(\Rd;\R^n)$
		and $[0,\infty)\ni t \mapsto \Psi(t)\in L^{r}(\Rd;\R^n)$ are continuously differentiable in $L^{p}(\Rd;\R^n)$ and $L^{r}(\Rd;\R^n)$, respectively,
		then $\Upsilon\cdot\Psi$ is continuously differentiable in $L^{\frac{pr}{p+r}}(\Rd)$ and $(\Upsilon\cdot\Psi)'= \Upsilon'\cdot \Psi+ \Upsilon\cdot\Psi'$.
	\end{lem}
	\begin{lem}\label{lem:otco}
	Let $p>1$. If $f,g\in \mathcal D_p(L)$, then for $t\in [0,\infty)$,
	$$\frac{\rm d}{\dt}\int_\Rd P_tf (P_tg)^{\langle p - 1 \rangle}\dx
=\int_\Rd \left( (P_tg)^{\langle p-1\rangle}LP_tf  +(p-1)P_tf|P_tg|^{p-2}LP_tg\right)\dx.$$
 For general $f,g\in L^p(\Rd)$, the formula holds for $t\in (0,\infty)$.
	\end{lem}
	\begin{proof}
	Of course, $P_tf$ and $P_tg$ are continuously differentiable at $t\geq0$ in $L^p(\Rd)$ and $\frac{\rm d}{\dt}P_tf =LP_tf$, $\frac{\rm d}{\dt}P_tg =LP_tg$.
	Hence, by Lemma~\ref{thm:difpow} with $n=1$, $(P_t g)^{\langle p-1\rangle}$ is continuously differentiable  at $t\geq0$ in $L^{\frac{p}{p-1}}(\Rd)$ and $\frac{\rm d}{\dt}(P_tg)^{\langle p - 1 \rangle}=(p-1)|P_tg|^{p-2}LP_tg$. By Lemma~\ref{lem:product} with $r=p/(p-1)$, $P_tf (P_tg)^{\langle p - 1 \rangle}$ is continuously differentiable at $t\geq0$ in $L^1(\Rd)$ and
	\begin{equation}\label{e.L1}
	\frac{\rm d}{\dt}\left(P_tf (P_tg)^{\langle p - 1 \rangle}\right)
=(P_tg)^{\langle p-1\rangle}LP_tf  +(p-1)P_tf|P_tg|^{p-2}LP_tg.
\end{equation}
 Since $u\mapsto \int_\Rd u(x)\,\dx$ is a continuous linear functional on $L^1(\Rd)$, we get the result (the case of arbitrary $f,g\in L^p(\Rd)$ follows since the semigroup $P_t$ is analytic). 
	\end{proof}
	
	\section{Convexity properties}\label{sec:apx}
	We provide here precise statements and proofs of convexity properties needed in Sections~\ref{sec:polarized} and \ref{sec:form-polarized}.
	First, we recall some facts from the theory of convex functions.

	Let $T\!:A\to\R$,  where the set $A\subset\R^n$ is convex.
	By definition, $d(w)\in\R^n$ is a \emph{subgradient} of $T$ at $w\in A$ if 
	\begin{align}
		\label{eq:subgrad}
		T(z) \ge T(w) + d(w)\cdot(z-w)
		\quad\mbox{for all }z\in A.
	\end{align}
	The function $T$ is convex in $A$ if and only if for every $w\in A$, a subgradient $d(w)$ exists.
	If $T$ is convex and the first-order partial derivatives of $T$ exist at some $w\in A$, then $T$ has exactly one subgradient at the point $w$, which is equal to its gradient $\nabla T(w)$.
	Denote by $\frac{\partial T}{\partial v}$  the directional derivative of $T$ along a given vector $v\in\R^n$. When $T$ is convex in $A$,
	we have that $d(w)$ is a subgradient of the function $T$ at the  point $w\in A$  if and only if 
	\begin{align*}
		\frac{\partial T}{\partial v} (w) \ge d(w)\cdot v,\quad v\in\R^n.
	\end{align*}
	For more details see Borwein and Lewis \cite[Chapter 3]{MR2184742}.
	
	We need the following lemma.
	\begin{lem}
		\label{lem:Yconv}
		Let $p\geq2$.
		The function
		\begin{align*}
			Y(z) := z_1 \left(z_2\right)^{p-1} + |z|^p,
			\quad z=(z_1,z_2)\in[0,\infty)^2,
		\end{align*}
		is convex on $[0,\infty)^2$.
	\end{lem}

	\begin{proof}
		As $Y$ is continuous on $[0,\infty)^2$, it is enough to prove the convexity on $(0,\infty)^2$.
		Recall 
		\eqref{eq:Nabla|z|^p}
		and
		\begin{align*}
			\nabla^2|z|^p
			=
			p(p-2)|z|^{p-4}
			\begin{bmatrix}
				z_1^2 & z_1 z_2 \\
				z_1z_2 & z_2^2
			\end{bmatrix}
			+
			p|z|^{p-2} \mbox{Id}
			,\quad
			z\in\R^2\setminus\{0\}.
		\end{align*}
		The Hessian $\nabla^2\left(z_1z_2^{p-1}\right)$ is calculated in \eqref{rem:z1z2pm1Hessian}.
		The Hessian  $\nabla^2 Y (z)$ of $Y$ is:
		\begin{align*}
			\begin{bmatrix}
				{p|z|^{p-2} + p(p-2)z_1^2|z|^{p-4} } & (p-1)z_2^{p-2} + p(p-2)z_1z_2|z|^{p-4} \\
				(p-1)z_2^{p-2} + p(p-2)z_1z_2|z|^{p-4} & (p-1)(p-2)z_1z_2^{p-3} + { p|z|^{p-2} + p(p-2)z_2^2|z|^{p-4} }
			\end{bmatrix}
		\end{align*}
		We will verify that for $z\in(0,\infty)^2$, the matrix is positive semi-definite.
		Clearly,
		$$\left[ p|z|^{p-2} + p(p-2)z_1^2|z|^{p-4} \right]>0.$$
		Moreover,
		after long, but elementary, calculations we get:
		\begin{eqnarray*}
			\det \nabla^2 Y (z)
			&=&
			\left[ p|z|^{p-2} + p(p-2)z_1^2|z|^{p-4} \right](p-1)(p-2)z_1z_2^{p-3}
			\\
			&&+
			p^2|z|^{2p-4} + p^2(p-2)|z|^{2p-4}
			\\
			&&-
			(p-1)^2z_2^{2p-4} - p(p-1)(p-2)z_1z_2^{p-1}|z|^{p-4}
			\\
			&=&
			p^2(p-1)|z|^{2p-4}-(p-1)^2 z_2^{2p-4}
			\\
			&&+
			p(p-1)(p-2)|z|^{p-4}\left((p-1)z_1^3 z_2^{p-3}-z_1z_2^{p-1}\right).
		\end{eqnarray*}
		We have $z_2\leq |z|$,  so applying Young's inequality with exponents $p$ and $q=p/(p-1)$ to the product $z_1z_2^{p-1}$  we obtain
		\begin{align*}
			z_1z_2^{p-1}
			&\leq
			\frac{z_1^p}{p} + \frac{(p-1)z_2^p}{p}
			=
			\frac{1}{p} (z_1^2)^\frac{p}{2} + \frac{p-1}{p} (z_2^2)^\frac{p}{2} \leq |z|^p
			.
		\end{align*}
		Summarizing,
		\begin{eqnarray*}
			\det \nabla^2 Y (z) &\geq& p(p-1)^2(p-2)|z|^{p-4}z_1^3z_2^{p-3}\\
			&&+ |z|^{2p-4}(p-1)\left(p^2-p(p-2)-(p-1)\right)\\
			&=& p(p-1)^2(p-2)|z|^{p-4}z_1^3z_2^{p-3} +|z|^{2p-4}(p-1)(p+1)>0.
		\end{eqnarray*}
	\end{proof}
	
	If $w_1\le z_1, w_2\le z_2, \ldots, w_k\le z_k$ implies $T(w_1,\ldots, w_k)\le T(z_1,\ldots,z_k)$ in the domain of a~real-valued function $T$, then we say $T$ is \emph{coordinate-wise nondecreasing}. The following fact is self-explanatory, see also Boyd and Vandenberghe \cite[Section 3.2.4]{MR2061575}.
	\begin{lem}
		\label{fact:convcomp}
		Let $S\!:A\to\R^k$, $S(A)\subset B$, and $T\!:B\to\R$, where $A\subset\R^n$ and $B\subset\R^k$ are convex. If each coordinate of $S$ is convex and $T$ is coordinate-wise nondecreasing  and convex, then the composition $T\circ S\!:A\to\R$ is convex.
	\end{lem}
	
	The following two lemmas are critical for our development.
	\begin{lem}
		\label{lem:Ypm}
		Let $p\geq2$ and define, for $z=(z_1,z_2)\in\R^2$,
		\begin{equation*}
			Y^{(+)}(z)  := z_1 \left((z_2)_+\right)^{p-1} + |z|^p, \qquad
			Y^{(-)}(z)  := z_1 \left((z_2)_-\right)^{p-1} + |z|^p. 
		\end{equation*}
		 The functions are convex on $[0,\infty)\times\R$.
	\end{lem}

	\begin{proof}
		Define $T\!:[0,\infty)\times\R\to[0,\infty)^2$ as
		\begin{align*}
			T(z) := (z_1, (z_2)_+),
			\quad z=(z_1,z_2)\in [0,\infty)\times\R,
		\end{align*}
		and let $Y\!:[0,\infty)^2\to\R$  be as in Lemma~\ref{lem:Yconv}. Since each coordinate of $T$ is convex and the function $Y$ is convex and coordinate-wise nondecreasing,  the composition
		\begin{align*}
			(Y\circ T)(z) = z_1 ((z_2)_+)^{p-1} + \left( (z_1)^2 + ((z_2)_+)^2 \right)^{p/2}
		\end{align*}
		is convex on $[0,\infty)\times\R$ (from Lemma~\ref{fact:convcomp}). Therefore,
		\begin{align*}
			Y^{(+)}(z) = \max\{(Y\circ T)(z),|z|^p\}
		\end{align*}
		is convex on $[0,\infty)\times\R$ as the maximum of convex functions \cite[Section 3.2.3]{MR2061575}.
		
		To prove the convexity of $Y^{(-)}$ we just notice that $Y^{(-)}(z_1,z_2)=Y^{(+)}(z_1,-z_2)$.
	\end{proof}
	
	\begin{lem}
		\label{lem:Jpppmmpositive}
		If $p>2$ then for all $z,w\in\R^2$, 
		\begin{align*}
			\mathcal{J}_p^{(++)}(w,z)+\mathcal{F}_p(w,z) \ge 0, \quad
			\mathcal{J}_p^{(-+)}(w,z)+\mathcal{F}_p(w,z) \ge 0,
		\end{align*}
		where
		$\mathcal{J}_p^{(++)}(w,z)$,  $\mathcal{J}_p^{(-+)}(w,z)$ are given by \eqref{eq:jotpp} and \eqref{eq:jotmp}.
	\end{lem}
	
	\begin{proof}
		Because of \eqref{eq:Fpsym} and \eqref{eq:Jpsym2}, we only need to show that $\mathcal{J}_p^{(++)}(w,z)+\mathcal{F}_p(w,z) \ge 0$. We rewrite this inequality as
		\begin{align}
			\label{eq:Yppsubgrad}
			Y^{(++)}(z)
			\ge
			Y^{(++)}(w) + d(w)\cdot(z-w),
		\end{align}
		where $Y^{(++)}(z):=(z_1)_+\left((z_2)_+\right)^{p-1} + |z|^p$ and
		\begin{align}\label{eq:111}
			d(w)
			:=
			\left(\indyk(w_1)\left((w_2)_+\right)^{p-1},
			(p-1)(w_1)_+\left((w_2)_+\right)^{p-2}\right)
			+
			pw^{\langle p - 1 \rangle}.
		\end{align}
		Therefore the proof of \eqref{eq:Yppsubgrad} amounts to checking that  $d(w)$ is a subgradient of the function $Y^{(++)}$ at the point $w\in\R^2$.
		
		To show \eqref{eq:Yppsubgrad},  we first establish the convexity of $Y^{(++)}$. Define $T\!:\R^2\to[0,\infty)^2$ as
		\begin{align*}
			T(z) := ((z_1)_+, (z_2)_+),
			\quad z=(z_1,z_2)\in \R^2.
		\end{align*}
		Let $Y\!:[0,\infty)^2\to\R$ as in Lemma~\ref{lem:Yconv}. Since each coordinate of $T$ is convex and the function $Y$ is convex and coordinate-wise nondecreasing, the convexity on $\R^2$ of the composition
		\begin{align*}
			(Y\circ T)(z) = (z_1)_+ ((z_2)_+)^{p-1} + \left( ((z_1)_+)^2 + ((z_2)_+)^2 \right)^{p/2}
		\end{align*}
		follows from Lemma~\ref{fact:convcomp}. Since
		\begin{align*}
			Y^{(++)}(z) = \max\{(Y\circ T)(z),|z|^p\},
		\end{align*}
		it  is convex on $\R^2$ as a maximum of two convex functions.
		
		If $w=0$ then $Y^{(++)}(w)=0$ and $d(w)=0$, hence \eqref{eq:Yppsubgrad} is true for every $z$.
		
		If $w\neq0$,  to show that $d(w)$ is a subgradient of $Y^{(++)}$ at $w$,  we need to prove that
		\begin{align*}
			\frac{\partial Y^{(++)}}{\partial v} (w) \ge d(w)\cdot v,
			\quad
			w\in\R^2\setminus\{{0}\},\quad \mbox{for every } v=(v_1,v_2)\in\R^2.
		\end{align*}
		Denote $B:=\{(w_1,w_2)\in\R^2\!: w_1=0, w_2>0\}$ as vertical positive semi-axis.
		The function $Y^{(++)}$ is differentiable everywhere but on $B$. Thus when $w\notin B$,   the gradient of $Y^{(++)}$ exists, is given by \eqref{eq:111}
		and
		\[
		\frac{\partial Y^{(++)}}{\partial v} (w) = \nabla Y^{(++)}(w)\cdot v = d(w)\cdot v.
		\]
		In the remaining case  $w\in B$, we have two possibilities.
		If $v_1\ge0$, then
		\begin{eqnarray*}
			\frac{\partial Y^{(++)}}{\partial v} (w)
			&=&
			\left((w_2)_+\right)^{p-1}v_1
			+
			pw^{\langle p - 1 \rangle}
			\cdot v
			\\
			&\ge&
			\frac{1}{2}\left((w_2)_+\right)^{p-1}v_1
			+
			pw^{\langle p - 1 \rangle}
			\cdot v
			=
			d(w)\cdot v.
		\end{eqnarray*}
		Otherwise, when $v_1<0$, then
		\begin{align*}
			\frac{\partial Y^{(++)}}{\partial v} (w)
			=
			pw^{\langle p - 1 \rangle}
			\cdot v
			\ge
			\frac{1}{2}\left((w_2)_+\right)^{p-1}v_1
			+
			pw^{\langle p - 1 \rangle}
			\cdot v
			=
			d(w)\cdot v.
		\end{align*}
		The proof is complete.
	\end{proof}

	\section{Alternative proof of polarization for $p\geq 3$}\label{s.ap}
	The main difficulty in the proof of Theorem~\ref{thm:polarized} above is to justify the limiting procedure in the absence of nonnegativity in the integrands.
	For $p\geq 3$, we can proceed differently: the absolute value of the function $\mathcal J_p$ is dominated by the function $\mathcal G_p$, which  helps with the integrability issues in the proof of the polarized Hardy\nobreakdash--Stein formula. 
	\begin{lem}
		\label{lem:Jpup}
		For every $p\ge3$, there is a constant $c_p>0$ such that
		\begin{align}
			\label{eq:Jpup}
			|\mathcal J_p(w,z)|\leq c_p \mathcal G_p(w,z) \asymp \mathcal H_p(w,z), \quad w,z\in \R^2.
		\end{align}
	\end{lem}
	\begin{proof}
		The formula \eqref{eq:Jp}, defining $\mathcal J_p$,  can be rewritten as
		\begin{eqnarray*}
			\mathcal{J}_p(w,z)
			&=&
			z_1z_2^{\langle p-1\rangle} - w_1w_2^{\langle p-1\rangle}-w_2^{\langle p-1\rangle}(z_1-w_1)-(p-1)w_1|w_2|^{p-2}(z_2-w_2)\\
			&=&
			(z_1-w_1)(z_2^{\langle p-1\rangle}-w_2^{\langle p-1\rangle})
			+
			w_1[z_2^{\langle p-1\rangle}-w_2^{\langle p-1\rangle} - (p-1)|w_2|^{p-2}(z_2-w_2)].
		\end{eqnarray*}
		Using \eqref{eq:difpowers} we can estimate the first summand above in the following manner
		\begin{eqnarray*}
			|z_1-w_1||z_2^{\langle p-1\rangle}-w_2^{\langle p-1\rangle}|
			&\leq&
			C_p' |w_1-z_1|\cdot|w_2-z_2|\cdot \left(|w_2|\vee|z_2|\right)^{p-2}
			\\
			&\leq&
			\frac{C_p'}{2}\left(|w_1-z_1|^2+|w_2-z_2|^2\right)
			\left(|w_2|\vee|z_2|\right)^{p-2}\\
			&\leq&
			\frac{C_p'}{2}|w-z|^2\left(|w|\vee|z|\right)^{p-2}
			=
			\frac{C_p'}{2} \mathcal{G}_p(w,z).
		\end{eqnarray*}
		For the second summand we use \eqref{eq:restfrench},
		\begin{align*}
			|w_1[z_2^{\langle p-1\rangle}-w_2^{\langle p-1\rangle} - & (p-1)|w_2|^{p-2}(z_2-w_2)]|
			\leq C_{p-1}''|w_1||z_2-w_2|^2(|w_2|\vee|z_2|)^{p-3}
			\\
			&\leq
			C_{p-1}''|w-z|^2 (|w_1|\vee |z_1|)(|w_2|\vee|z_2|)^{p-3}\\
			&\leq C_{p-1}''|w-z|^2(|w|\vee |z|)^{p-2}
			=
			C_{p-1}''\mathcal{G}_p(w,z).
		\end{align*}
		Thus,
		\begin{align*}
			|\mathcal{J}_p(w,z)| \le c_p \mathcal{G}_p(w,z)
		\end{align*}
		and $\mathcal{G}_p(w,z)\asymp\mathcal{H}_p(w,z)$ by \eqref{eq:Hp2Gp2}.
	\end{proof}
	
	\begin{rem}
		The statement \eqref{eq:Jpup} stays true also for $p=2$. Indeed, by \eqref{eq:Jp2},
		\begin{align*}
			|\mathcal{J}_2(w,z)|
			=
			|(z_1-w_1)(z_2-w_2)|
			\leq
			|z-w|^2
			=
			\mathcal{G}_2(w,z).
		\end{align*}
		On the other hand, it fails in general  for $p\in(1,3)\setminus\{2\}$. Indeed, for $k=1,2,\ldots$, let $w^{(k)}:=\left(1,\frac{1}{k}\right)$, $z^{(k)}:=\left(1,\frac{2}{k}\right)$. Then, by \eqref{eq:Jp},
		\begin{eqnarray*}
			|\mathcal J_p(w^{(k)},z^{(k)})|
			&=&
			\left|
			\frac{2^{p-1}}{k^{p-1}} - \frac{1}{k^{p-1}} - \frac{p-1}{k^{p-1}}
			\right|
			=
			\left| 2^{p-1} - p \right| \frac{1}{k^{p-1}}, \\
			\mathcal G_p(w^{(k)},z^{(k)})
			&=&
			\frac{1}{k^2} \left(1+\frac{4}{k^2}\right)^\frac{p-2}{2}.
		\end{eqnarray*}
		Our claim is verified by notifying that
		\begin{align*}
			\frac{|\mathcal J_p(w^{(k)},z^{(k)})|}{\mathcal{G}_p(w^{(k)},z^{(k)})}
			=
			\frac{k^{3-p} \left| 2^{p-1} - p \right|}{\left(1+\frac{4}{k^2}\right)^\frac{p-2}{2}}
			\to
			\infty
			\quad \mbox{as } k\to\infty.
		\end{align*}
	\end{rem}
	
	Estimate \eqref{eq:Jpup} permits to substantially simplify the proof of the polarized Hardy\nobreakdash--Stein identity \eqref{eq:HS2v2}. Indeed, for $f,g\in\Lp$, $u(t)=P_tf$, $v(t)=P_tg$, $\Phi=(u,v)$, and $p\geq 3$, from Theorem~\ref{thm:HS15} we have that
	\begin{align}\label{eq:Hp-int}
		\int_0^\infty\int_{\Rd}\int_{\Rd} \mathcal H_p(P_t\Phi(x),P_t\Phi(y))\nu(x,y)\,\dx\dy\dt <\infty,
	\end{align}
	so in view of Lemma~\ref{lem:Jpup}, an analogous integral of $\left|\mathcal J_p(P_t\Phi(x),P_t\Phi(y))\right|$ is convergent as well.
	We next review 
the proof of Theorem~\ref{thm:polarized}:
	for $f,g\in\mathcal D_p(L)$, we differentiate $u(t)v(t)^{\langle p - 1 \rangle}$ in $L^1(\Rd)$ as in \eqref{eq:diff-l1} and we have:
	\[  \frac{\rm d}{\dt} \intRd u(t) v(t)^{\langle p - 1 \rangle}\,\dx
	=
	-\lim_{h\to0^+}\intRd\intRd
	\mathcal{J}_p(P_t\Phi(x),P_t\Phi(y))
	\frac{p_h(x,y)}{h}\,\dx\dy.\]
	Since $|\mathcal J_p|\leq \mathcal H_p$ and the integral in \eqref{eq:Hp-int} is convergent, we can pass to the limit when  $h\to 0^+$ (we use the Dominated Convergence Theorem, \eqref{eq:mp}, and \eqref{eq:zp}) to obtain
	\begin{equation}\label{e.tjJ} \frac{\rm d}{\dt} \intRd u(t) v(t)^{\langle p - 1 \rangle} \,\dx
	=
	-\intRd\intRd
	\mathcal{J}_p(P_t\Phi(x),P_t\Phi(y))
	\nu(x,y) \,\dx\dy.
	\end{equation}
	The rest of the proof remains unchanged: we integrate from
	0 to $T$ with $T>0$ fixed, then we pass to the limit $T\to\infty$.
	Then, we relax the assumption that $f,g\in \mathcal D_p(L)$ by using the analyticity of the semigroup.
	
	\section{Proof of \eqref{eq:HSgauss2}}
	\label{sec:brown}

	Let $\{ B_t, t\geq0\}$ be the Brownian motion on the Euclidean space $\Rd$ running at twice the usual speed, and let $(P_t)_{t\geq 0}$ be
its semigroup:
	\begin{align*}
		P_t f(x):=\Ex f(B_t)= \intRd f(y) p_t(x,y) \,\dy = (p_t\ast f)(x),
		\quad t>0,x\in\Rd,
	\end{align*}
	where
	\begin{align*}
		p_t(x) = (4\pi t)^{-d/2} e^{-\frac{|x|^2}{4t}},
		\quad t>0,x\in\Rd
	\end{align*}
	and $p_t(x,y):=p_t(x-y)$,  as before. Let $1<p<\infty$. It is well known that $(P_t)_{t\geq 0}$ is a strongly continuous, analytic,  Markovian semigroup of symmetric operators in $\Lp$.  In particular, for every $t>0$ and $f\in\Lp$,  $P_tf$ belongs to the domain of the generator of this semigroup. Estimates \eqref{eq:Ptf} and \eqref{eq:stein} hold true as well, therefore the key ingredients needed to prove Hardy\nobreakdash--Stein identity remain satisfied for the Brownian motion. Thus, for every $u\in \Lp$, we define, as before,
	\begin{align*}
		\E_p[u] := \lim_{t\to 0} \E^{(t)}(u,u^{\langle p-1 \rangle})
	\end{align*}
	and
	{
		\begin{align*}
			\DEp
			&=
			\{u\in \Lp: \text{ finite } \lim_{t\to 0}\mathcal E^{(t)}(u,u^{\langle p-1 \rangle}) \text{ exists}\}.
		\end{align*}
	}

	Similarly as in the proof of 
	Theorem~\ref{thm:HS15}, 
	we obtain
	\begin{align}
		\label{eq:HSgauss}
		\intRd |f|^p \,\dx = p \int_0^\infty \E_p[P_t f] \,\dt
		,\quad  f\in\Lp.
	\end{align}
	The generator of the Gaussian semigroup $(P_t)_{t\geq 0}$ acting on $u\in\Lp$ is
	\begin{align*}
		L u := 
		\lim_{h\to0^+} \frac{1}{h}(P_h u - u),
		\quad \textit{if the limit exists in }\Lp.
	\end{align*}
	We can also write
	\begin{align*}
		L u = \sum_{j=1}^d \frac{\partial^2 u}{\partial x_j^2},\quad u\in\Lp,
	\end{align*}
	where the partial derivatives of $u$ are understood in the distributional sense. We kept the letter $L$ here, to be in accordance with the previous development.
	The domain of the generator is
	\begin{eqnarray*}
		\mathcal D_p(L)
		&:=&
		\{u\in\Lp: \lim_{h\to0^+} (P_h u - u)/h \text{ exists in }\Lp\}
		\\
		&\ =&
		\left\{u\in\Lp: \sum_{j=1}^d \frac{\partial^2 u}{\partial x_j^2}\in\Lp\right\}.
	\end{eqnarray*}
	In  Appendix~\ref{sec:app} we explain and justify the above statements.
	
	As earlier, for $u\in\mathcal D_p(L)\subset\DEp$,
	\begin{align}
		\label{eq:pfaggauss}
		\E_p[u] = -\langle Lu, u^{\langle p-1 \rangle} \rangle.
	\end{align}
	To express the Hardy\nobreakdash--Stein identity in a more explicit form, we need the following identity, which was proved by Metafune and Spina \cite{MR2465581}.
	\begin{lem}
		Let $1<p<\infty$. For $u\in W^{2,p}(\Rd)$,
		\begin{align}
			\label{eq:metafune}
			\intRd u^{\langle p-1 \rangle} L u \,\dx
			=
			-(p-1)\intRd |u|^{p-2}|\nabla u|^2 \,\dx,
		\end{align}
		where $W^{k,p}(\Rd)$ is the Sobolev space of order $k$.
	\end{lem}
	It is not hard to see that for $t>0$ and $f\in\Lp$, we have $P_tf\in W^{2,p}(\Rd)$. Indeed, for every multi-index $\alpha=(\alpha_1,\ldots,\alpha_d)$,  we denote, as usual, $|\alpha|:=\alpha_1+\ldots+\alpha_d$ and
		$\partial^\alpha := \frac{\partial^{|\alpha|}}{\partial x_1^{\alpha_1}\ldots\partial x_d^{\alpha_d}}$.
	Then,
	\begin{align*}
		\left\| \partial^\alpha P_t f \right\|_{\Lp}
		=
		\left\| (\partial^\alpha p_t) \ast f \right\|_{\Lp}
		\leq
		\left\| \partial^\alpha p_t\right\|_{L^1(\Rd)} \cdot \left\| f \right\|_{\Lp}
		<\infty.
	\end{align*}
	By \eqref{eq:pfaggauss} and \eqref{eq:metafune}, for $f\in\Lp$ and $t>0$, 
	\begin{align}
		\label{eq:metafuneEp}
		\E_p[P_t f] = -\langle \Delta P_t f, (P_t f)^{\langle p-1 \rangle} \rangle
		=
		(p-1)\intRd |P_t f|^{p-2} |\nabla P_t f|^2 \,\dx.
	\end{align}
	Since $P_t f\in C^\infty(\Rd)$, the above derivatives are taken in the classical sense. Here $\Delta$ is the classical Laplacian. Using \eqref{eq:metafuneEp}, we can express Hardy\nobreakdash--Stein identity for the Gaussian semigroup \eqref{eq:HSgauss} in the desired form.
	This finishes the proof of \eqref{eq:HSgauss2}.

	\section{The  generator of the Gaussian semigroup in $L^p$}
	\label{sec:app}

	For completeness we prove the equivalence of two definitions of the Laplacian on $\Lp$ used in Appendix~\ref{sec:brown}.
		Let $1<p<\infty$. Let $(P_t)_{t \geq 0}$ be the Gaussian semigroup. It is well known that for $\varphi\in C_c^\infty (\Rd),$  the infinitely differentiable functions on $\Rd$ with compact support,
	\begin{align*}
		 \lim_{h\to0^+} \frac{1}{h}(P_h \varphi - \varphi)  = \sum_{j=1}^d \frac{\partial^2 \varphi}{\partial x_j^2} 
		=:\Delta\varphi  
		\quad \text{ (the limit taken in }\Lp ).
	\end{align*} We show that the equality is satisfied also for those functions from $\Lp$, for which the right-hand side limit exists, without further regularity assumptions.
	
	The {\it semigroup Laplacian} is defined as
	\begin{align}
		L f := \lim_{h\to0^+} \frac{1}{h}(P_h f - f)
		,\quad f\in\mathcal D_p(L)\subset\Lp,
	\end{align}
	where the limit above is taken in $\Lp$, for $f$ in the natural domain:
	\begin{align*}
		\D_p(L)
		:=
		\{u\in\Lp: \lim_{h\to0^+} (P_h u - u)/h \text{ exists in }\Lp\}.
	\end{align*}
	Since $L$ is the generator of a strongly continuous semigroup in $\Lp$,  the operator $\lambda I - L\!:\D_p(L)\to\Lp$ is a bijection for every $\lambda>0$.

	We then recall the notion of the {\em distributional Laplacian} $\tilde L$ of $f\in\Lp$. 
If there exists $g\in\Lp$ such that
	\begin{align*}
		\langle g, \varphi \rangle = \langle f, L\varphi \rangle
		= \left\langle f, \sum_{j=1}^d \frac{\partial^2 \varphi}{\partial x_j^2} \right\rangle = \langle f, \Delta\varphi \rangle
	\end{align*}
	for all test functions $\varphi\in C_c^\infty (\Rd)$, then we let $\tilde{L}f := g$. The class of all such functions $f$
is denoted $\D(\tilde{L})$. In other words,
	\begin{align*}
		\tilde{L}f = \sum_{j=1}^d \frac{\partial^2 f}{\partial x_j^2},
	\end{align*}
	where the partial derivatives are taken in distributional sense.
The operators $L$ and $\tilde L$ coincide, which we prove below.
	\begin{lem}
		\label{lem:Lext}
		The operator $\tilde{L}$ is an extension of $L$.
	\end{lem}
	\begin{proof}
		Let $f\in\D_p(L)$. For every $\varphi\in C_c^\infty (\Rd)$, by symmetry of the operators $P_h$,
		\begin{align*}
			\langle \tilde{L}f, \varphi \rangle
			=
			\langle f, L\varphi \rangle
			=
			\lim_{h\to0^+} \frac{1}{h} \left(\langle f, P_h \varphi \rangle - \langle f, \varphi \rangle \right)
			=
			\lim_{h\to0^+} \frac{1}{h} \left(\langle P_h f, \varphi \rangle - \langle f, \varphi \rangle \right)
			=
			\langle Lf, \varphi \rangle.
		\end{align*}
		Thus, $f\in\D(\tilde{L})$ and $\tilde{L}f=Lf$.
	\end{proof}
	
	\begin{lem}
		\label{lem:Lsurj}
		For every $\lambda>0$, the operator $\lambda I - \tilde{L}$ defined on $\D(\tilde{L})$ is one-to-one.
	\end{lem}
	\begin{proof}
		Assume that $f\in\D(\tilde{L})$ and $\lambda f - \tilde{L}f=0$. We prove that $f=0.$ Take $\varphi\in C_c^\infty (\Rd)$,  then using properties of convolutions and distributional derivatives, we can write
		\begin{align}
			\label{eq:astphi}
			\lambda f\ast\varphi - \tilde{L}(f\ast\varphi)
			=
			\lambda f\ast\varphi - (\tilde{L}f)\ast\varphi
			=
			(\lambda f - \tilde{L}f) \ast \varphi
			=
			0 \ast \varphi
			=
			0.
		\end{align}
		This yields $f\ast\varphi=0.$ Indeed, assuming the contrary, 
		since $f\ast\varphi \in C_0^\infty (\Rd)$, there is a point $x_0\in\Rd$ which is the positive maximum or the negative minimum of $f\ast\varphi$. If $x_0$ is the positive maximum of $f\ast\varphi$, then $\tilde{L}(f\ast\varphi)(x_0)=\Delta(f\ast\varphi)(x_0)\leq0$ and
		\begin{align*}
			0=  (\lambda f\ast\varphi)(x_0) - \tilde{L}(f\ast\varphi)(x_0)
			\geq
			(\lambda f\ast\varphi)(x_0)
			>
			0,
		\end{align*}
		a contradiction. The case of $s_0$ being the negative minimum is handled similarly. 
Therefore $f\ast\varphi=0$ for every $\varphi\in C_c^\infty (\Rd)$, meaning $f=0$.
	\end{proof}
	
		\begin{prop}\label{prop:Laplace}
		We have  $\D_p(L)=\D(\tilde{L})$
		and $L=\tilde{L}$.
	\end{prop}
	\begin{proof}
		Take any $\lambda>0$. In view of Lemma~\ref{lem:Lext} and \ref{lem:Lsurj}, the bijection\linebreak $\lambda I - L\!:\D_p(L)\to\Lp$ and its injective extension  $\lambda I -\tilde{L}\!:\D(\tilde{L})\to\Lp$ are equal,
		$\D_p(L)=\D(\tilde{L})$, $\lambda I - L=\lambda I -\tilde{L}$. Thus, $L=\tilde{L}$.
	\end{proof}


\begin{thebibliography}{10}

\bibitem{MR3556449}
R.~Ba\~{n}uelos, K.~Bogdan, and T.~Luks.
\newblock Hardy--{S}tein identities and square functions for semigroups.
\newblock {\em J. Lond. Math. Soc. (2)}, 94(2):462--478, 2016.

\bibitem{MR3994925}
R.~Ba{\~n}uelos and D.~Kim.
\newblock Hardy--{Stein} identity for non-symmetric {L{\'e}vy} processes and
  {Fourier} multipliers.
\newblock {\em J. Math. Anal. Appl.}, 480(1):20, 2019.
\newblock Id/No 123383.

\bibitem{MR2048514}
G.~Barbatis, S.~Filippas, and A.~Tertikas.
\newblock A unified approach to improved {$L^p$} {H}ardy inequalities with best
  constants.
\newblock {\em Trans. Amer. Math. Soc.}, 356(6):2169--2196, 2004.

\bibitem{MR3251822}
K.~Bogdan, B.~Dyda, and T.~Luks.
\newblock On {H}ardy spaces of local and nonlocal operators.
\newblock {\em Hiroshima Math. J.}, 44(2):193--215, 2014.

\bibitem{MR4600287}
K.~Bogdan, D.~Fafu\l{}a, and A.~Rutkowski.
\newblock The {D}ouglas formula in {$L^p$}.
\newblock {\em NoDEA Nonlinear Differential Equations Appl.}, 30(4):Paper No.
  55, 2023.

\bibitem{MR4088505}
K.~Bogdan, T.~Grzywny, K.~Pietruska-Pa\l{}uba, and A.~Rutkowski.
\newblock Extension and trace for nonlocal operators.
\newblock {\em J. Math. Pures Appl. (9)}, 137:33--69, 2020.

\bibitem{MR4589708}
K.~Bogdan, T.~Grzywny, K.~Pietruska-Pa\l{}uba, and A.~Rutkowski.
\newblock Nonlinear nonlocal {D}ouglas identity.
\newblock {\em Calc. Var. Partial Differential Equations}, 62(5):Paper No. 151,
  31, 2023.

\bibitem{MR3165234}
K.~Bogdan, T.~Grzywny, and M.~Ryznar.
\newblock Density and tails of unimodal convolution semigroups.
\newblock {\em J. Funct. Anal.}, 266(6):3543--3571, 2014.

\bibitem{MR4372148}
K.~Bogdan, T.~Jakubowski, J.~Lenczewska, and K.~Pietruska-Pa\l{}uba.
\newblock Optimal {H}ardy inequality for the fractional {L}aplacian on {$L^p$}.
\newblock {\em J. Funct. Anal.}, 282(8):Paper No. 109395, 31, 2022.

\bibitem{2024+KB-DK-KPP}
K.~Bogdan, D.~Kutek, and K.~Pietruska-Pa\l{}uba.
\newblock Bregman variation of semimartingales.
\newblock arXiv:2412.18345.

\bibitem{MR4521651}
K.~Bogdan and M.~Wi\c{e}cek.
\newblock Burkholder inequality by {B}regman divergence.
\newblock {\em Bull. Pol. Acad. Sci. Math.}, 70(1):83--92, 2022.

\bibitem{2020arXiv200703674B}
M.~{Bonforte}, J.~{Dolbeault}, B.~{Nazaret}, and N.~{Simonov}.
\newblock {Stability in Gagliardo--Nirenberg--Sobolev inequalities: flows,
  regularity and the entropy method}.
\newblock To appear in Memoirs AMS, arXiv:2007.03674.

\bibitem{MR4518648}
M.~Borowski and I.~Chlebicka.
\newblock Controlling monotonicity of nonlinear operators.
\newblock {\em Expo. Math.}, 40(4):1159--1180, 2022.

\bibitem{MR2184742}
J.~M. Borwein and A.~S. Lewis.
\newblock {\em Convex analysis and nonlinear optimization}, volume~3 of {\em
  CMS Books in Mathematics/Ouvrages de Math\'{e}matiques de la SMC}.
\newblock Springer, New York, second edition, 2006.
\newblock Theory and examples.

\bibitem{MR2061575}
S.~Boyd and L.~Vandenberghe.
\newblock {\em Convex optimization}.
\newblock Cambridge University Press, Cambridge, 2004.

\bibitem{MR1853037}
J.~A. Carrillo, A.~J\"{u}ngel, P.~A. Markowich, G.~Toscani, and A.~Unterreiter.
\newblock Entropy dissipation methods for degenerate parabolic problems and
  generalized {S}obolev inequalities.
\newblock {\em Monatsh. Math.}, 133(1):1--82, 2001.

\bibitem{MR3646773}
W.~Cygan, T.~Grzywny, and B.~Trojan.
\newblock Asymptotic behavior of densities of unimodal convolution semigroups.
\newblock {\em Trans. Amer. Math. Soc.}, 369(8):5623--5644, 2017.

\bibitem{MR1103113}
E.~B. Davies.
\newblock {\em Heat kernels and spectral theory}, volume~92 of {\em Cambridge
  Tracts in Mathematics}.
\newblock Cambridge University Press, Cambridge, 1990.

\bibitem{MR1808433}
W.~Farkas, N.~Jacob, and R.~L. Schilling.
\newblock Feller semigroups, {$L^p$}-sub-{M}arkovian semigroups, and
  applications to pseudo-differential operators with negative definite symbols.
\newblock {\em Forum Math.}, 13(1):51--90, 2001.

\bibitem{MR2778606}
M.~Fukushima, Y.~Oshima, and M.~Takeda.
\newblock {\em Dirichlet forms and symmetric {M}arkov processes}, volume~19 of
  {\em De Gruyter Studies in Mathematics}.
\newblock Walter de Gruyter \& Co., Berlin, extended edition, 2011.

\bibitem{MR4622410}
M.~Gutowski.
\newblock Hardy--{S}tein identity for pure-jump {D}irichlet forms.
\newblock {\em Bull. Pol. Acad. Sci. Math.}, 71(1):65--84, 2023.

\bibitem{gutowski2023beurlingdeny}
M.~Gutowski and M.~Kwaśnicki.
\newblock Beurling--{D}eny formula for {Sobolev-Bregman} forms, 2023.
\newblock arXiv:2312.10824.

\bibitem{MR1923834}
S.-Z. Huang.
\newblock Inequalities for submarkovian operators and submarkovian semigroups.
\newblock {\em Math. Nachr.}, 243:75--91, 2002.

\bibitem{MR1873235}
N.~Jacob.
\newblock {\em Pseudo differential operators and {M}arkov processes. {V}ol.
  {I}}.
\newblock Imperial College Press, London, 2001.
\newblock Fourier analysis and semigroups.

\bibitem{MR1694522}
M.~Langer and V.~Maz'ya.
\newblock On {$L^p$}-contractivity of semigroups generated by linear partial
  differential operators.
\newblock {\em J. Funct. Anal.}, 164(1):73--109, 1999.

\bibitem{MR1224619}
V.~A. Liskevich and M.~A. Perel'muter.
\newblock Analyticity of sub-{M}arkovian semigroups.
\newblock {\em Proc. Amer. Math. Soc.}, 123(4):1097--1104, 1995.

\bibitem{MR1409835}
V.~A. Liskevich and Y.~A. Semenov.
\newblock Some problems on {M}arkov semigroups.
\newblock In {\em Schr\"{o}dinger operators, {M}arkov semigroups, wavelet
  analysis, operator algebras}, volume~11 of {\em Math. Top.}, pages 163--217.
  Akademie Verlag, Berlin, 1996.

\bibitem{MR2465581}
G.~Metafune and C.~Spina.
\newblock An integration by parts formula in {S}obolev spaces.
\newblock {\em Mediterr. J. Math.}, 5(3):357--369, 2008.

\bibitem{MR2400106}
Y.~Pinchover, A.~Tertikas, and K.~Tintarev.
\newblock A {L}iouville-type theorem for the {$p$}-{L}aplacian with potential
  term.
\newblock {\em Ann. Inst. H. Poincar\'{e} C Anal. Non Lin\'{e}aire},
  25(2):357--368, 2008.

\bibitem{MR1739520}
K.-i. Sato.
\newblock {\em L\'{e}vy processes and infinitely divisible distributions},
  volume~68 of {\em Cambridge Studies in Advanced Mathematics}.
\newblock Cambridge University Press, Cambridge, 1999.
\newblock Translated from the 1990 Japanese original, Revised by the author.

\bibitem{MR1759788}
I.~Shafrir.
\newblock Asymptotic behaviour of minimizing sequences for {H}ardy's
  inequality.
\newblock {\em Commun. Contemp. Math.}, 2(2):151--189, 2000.

\bibitem{MR0252961}
E.~M. Stein.
\newblock {\em Topics in harmonic analysis related to the {L}ittlewood--{P}aley
  theory}.
\newblock Annals of Mathematics Studies, No. 63. Princeton University Press,
  Princeton, N.J.; University of Tokyo Press, Tokyo, 1970.

\bibitem{MR0803094}
N.~T. Varopoulos.
\newblock Hardy--{L}ittlewood theory for semigroups.
\newblock {\em J. Funct. Anal.}, 63(2):240--260, 1985.

\end{thebibliography}

\end{document}